\documentclass[%
  onecolumn,
]{mpi2015-cscpreprint}

\usepackage{amsmath, amscd, amssymb, amsthm, mathrsfs, mathtools, bm}
\usepackage{array, booktabs, makecell} 
\newcolumntype{P}[1]{>{\centering\arraybackslash}p{#1}}
\newcolumntype{M}[1]{>{\centering\arraybackslash}m{#1}}
\usepackage[american]{babel}
\usepackage{enumerate}
\usepackage{graphicx}
\usepackage{hyperref}
\usepackage{soul}
\usepackage{tikz}
\usepackage[backgroundcolor=violet!60,textcolor=white,linecolor=gray]{todonotes}
\usepackage{xspace} 
\usepackage{multirow}
\usepackage{adjustbox}

\usepackage[dvipsnames]{xcolor}
\usepackage[mathscr]{euscript}  

\usepackage[algosection,vlined,algo2e,linesnumbered,algoruled]{algorithm2e}
\SetKwComment{Comment}{}{}
\SetKwProg{Fn}{Function}{}{end}
\DontPrintSemicolon{}

\SetKwInput{KwParameters}{Parameters}
\newcommand{\SComment}[1]{{\smaller{\tcp*[r]{#1}}}}

\numberwithin{equation}{section}
\numberwithin{table}{section}
\numberwithin{figure}{section}

\theoremstyle{plain}
    \newtheorem{theorem}{Theorem}
    
    \newtheorem{lemma}[theorem]{Lemma}
    
\theoremstyle{definition}

\theoremstyle{remark}

\numberwithin{theorem}{section}

\DeclareMathOperator{\diag}{diag}
\DeclareMathOperator{\sign}{sign}   

\newcommand{\R}{\mathbb{R}}   
\newcommand{\bigO}[1]{\mathcal{O}\left(#1\right)}       

\newcommand{\norm}[1]{\left\lVert#1\right\rVert}

\DeclareMathOperator{\fl}{\texttt{fl}}

\newcommand{\wh}[1]{\widehat{#1}}

\newcommand{\uhigh}{u_{h}}
\newcommand{\ulow}{u_{\ell}}

\definecolor{plotblue}{RGB}{0, 113.9850, 188.9550}
\definecolor{plotred}{RGB}{216.7500, 82.8750, 24.9900}
\definecolor{plotpurple}{RGB}{125.9700, 46.9200, 141.7800}

\begin{document}


\title{Reduced and Mixed-Precision Block Householder QR Factorization}
 
\author[$\dagger$]{Peter Benner}
\affil[$\dagger$]{Computational Methods in Systems and Control Theory, Max Planck Institute for Dynamics of Complex Technical Systems, Sandtorstr.\ 1, 39106 Magdeburg, Germany\authorcr
	\email{benner@mpi-magdeburg.mpg.de}, \orcid{0000-0003-3362-4103}}

\author[$\dagger$]{Martin K\"{o}hler}
\affil[$\dagger$]{Computational Methods in Systems and Control Theory, Max Planck Institute for Dynamics of Complex Technical Systems, Sandtorstr.\ 1, 39106 Magdeburg, Germany\authorcr
	\email{koehlerm@mpi-magdeburg.mpg.de}, \orcid{0000-0003-2338-9904}}

\author[$\dagger$]{Xiaobo Liu}
\affil[$\dagger$]{Computational Methods in Systems and Control Theory, Max Planck Institute for Dynamics of Complex Technical Systems, Sandtorstr.\ 1, 39106 Magdeburg, Germany\authorcr
	\email{xliu@mpi-magdeburg.mpg.de}, \orcid{0000-0001-8470-8388}}

\author[$\dagger$]{Eda Oktay}
\affil[$\dagger$]{Computational Methods in Systems and Control Theory, Max Planck Institute for Dynamics of Complex Technical Systems, Sandtorstr.\ 1, 39106 Magdeburg, Germany\authorcr
    \email{oktay@mpi-magdeburg.mpg.de}, \orcid{0000-0003-0761-2184}}

\shortauthor{P. Benner, M. K\"{o}hler, X. Liu, E. Oktay}
  

  
\abstract{
    Recent advances in hardware architecture have highlighted the importance of reduced-precision arithmetic for improving performance and reducing energy consumption in large-scale scientific applications. However, uniformly low-precision methods can suffer from numerical instability, especially in fundamental linear algebra routines. Mixed-precision strategies therefore play an important role in balancing efficiency, stability, and accuracy. In this paper, we focus on the block Householder QR decomposition using the WY representation and study its behavior in reduced- and mixed-precision settings. Starting from a uniform half-precision variant, we show that performing norm computations with fp32 accumulation is necessary to avoid overflow and underflow on large matrices. Based on this observation, we introduce mpWYQR, a mixed-precision algorithm that increases the precision of only selected computations and avoids most of the precision conversions required by the existing mixed-precision blocked QR algorithm (mpBQR3) introduced in [L. M. Yang, A. Fox, and G. Sanders, \textit{SIAM J. Sci. Comp.}, 43.3 (2021), pp. A1723--A1753]. Our analysis shows that mpWYQR accuracy is governed by low-precision unit roundoff, and that its error bound exceeds that of mpBQR3 by only a constant factor, with identical heuristic bounds. In experiments across a range of matrix sizes and condition numbers, we show that mpWYQR retains mpBQR3 accuracy while running up to $18\times$ faster and using less memory. We also discuss the performance cost of computing inner products in higher precision, and the library support needed for an efficient GPU implementation.
}

\novelty{
    A mixed-precision Householder QR algorithm based on the WY representation, mpWYQR, is introduced, in which the precision of only selected computations is raised, and fewer precision conversions are required than in existing mixed-precision blocked QR algorithms. Its error bound differs from that of mpBQR3 only by a constant factor, and it achieves comparable accuracy experimentally with speedups of up to $18\times$. These results are important for using low-precision hardware in QR-based solvers in high-performance computing.
}

\maketitle

\section{Introduction}
The QR decomposition is a fundamental and versatile tool in numerical linear algebra, with various applications including least-squares problems, eigenvalue computations, and low-rank matrix approximations \cite{GolV13}. It is a key computational step in many scientific and engineering applications, where both accuracy and performance are critical. Among the various algorithms for computing the QR factorization, Householder QR \cite{Hou58} is widely used owing to its normwise backward stability, a desirable property in the presence of rounding errors. However, classical Householder QR can suffer from communication and synchronization bottlenecks on modern parallel architectures: the repeated global synchronizations entailed in the sequential construction of Householder reflectors limit its scalability, and the computational cost is \emph{cubic} in the matrix dimension. 
As problem sizes continue to grow in modern scientific applications, performing QR factorization entirely in high precision becomes computationally expensive, motivating more efficient alternatives.

To address these limitations, several improvements have been proposed over the years. Bischof and Van Loan introduced the WY representation of Householder transformations, enabling the use of BLAS3-level operations and improving data locality  \cite{BisV87}. Rotella and Zambettakis extended this idea through block Householder transformations, enabling more efficient parallel QR factorization \cite{RotZ99}. Demmel et al. further developed communication-avoiding, parallel, and stable QR algorithms based on Householder reflections, reducing the communication overhead while preserving numerical stability \cite{DemGHetal12}. More recently, Haidar et al. demonstrated the effectiveness of batched and GPU-resident Householder QR factorization, enabling high-performance computing on modern architectures \cite{HaiDTetal15}.

With the emergence of low-precision arithmetic, particularly half-precision (binary16, referred to as fp16)\cite{ieee19}, on modern hardware, both computational time and memory bandwidth requirements can be significantly reduced. Using low precision uniformly throughout the QR factorization nevertheless introduces significant numerical challenges. In particular, computing Householder reflectors in fp16 is prone to overflow and underflow when evaluating the $2$-norm of a vector, due to the limited dynamic range of fp16. Even when the range issue is avoided, the loss of orthogonality (LOO) in the computed Q-factor is in general bounded by the unit roundoff of fp16, which is about $4.8\times10^{-4}$, and achieving even this level of orthogonality reliably requires the use of higher precision in critical parts of the factorization.
While this accuracy is insufficient for many numerical linear algebra applications, some application settings may accept a Q-factor with fp16-level LOO. 
For instance, in certain machine learning applications such as Gaussian process training, half-precision iterative solvers that maintain orthogonality only at the level of the working precision have been shown to perform well in practice \cite{MadPW22}. Similarly, in randomized numerical linear algebra, where QR is applied to sketched matrices or used within randomized range-finding algorithms \cite{HMT11}, the inherent randomness already introduces errors of a comparable or larger magnitude, making fp16 arithmetic a reasonable choice. 
Furthermore, in signal processing applications such as adaptive beamforming \cite{RSVN01} and wireless multiple‐input multiple‐output precoding \cite{CHL19}, where the required application-level accuracy may tolerate reduced-precision arithmetic, a lower-precision QR factorization may be sufficient in practice.

Higham provided a rigorous rounding-error analysis of Householder QR in uniform precision, showing that, for $A\in\mathbb R^{m\times n}$, the worst-case backward error is of order $mn$ times the working-precision unit roundoff \cite{Hig02}. Building on these analyses, recent advances have improved mixed-precision variants of Householder QR. Yang et al. proposed a mixed-precision block Householder QR algorithm (mpBQR3), which achieves improved performance by performing selected operations in lower precision while attaining an LOO bound of order $2u_l + nu_h$, where $u_l$ and $u_h$ denote the lower- and higher-precision unit roundoffs, respectively \cite{YanFS21}.
More recently, Leng et al. developed a high-performance Householder QR factorization for NVIDIA GPUs that exploits Tensor Cores and reduced-precision arithmetic \cite{LZWW25}.
Connolly and Higham developed a probabilistic rounding-error analysis for Householder QR that yields a probabilistic bound proportional to $\sqrt{mn}u$ with high probability under the assumptions that rounding errors are mean-independent and of mean zero \cite{ConH23}. A comprehensive survey of mixed-precision methods in numerical linear algebra, including orthogonalization and QR factorization, can be found in \cite{HigM22}.

Beyond these direct use cases, low-precision QR factorization can also be desirable when a correction step can recover lost accuracy. For example, in iterative refinement methods, a low-precision QR factorization can be used to construct a preconditioner, with the accuracy subsequently improved through refinement iterations \cite{CarHP20}. Similarly, in mixed-precision Jacobi SVD algorithms, a low-precision QR factorization serves as a preprocessing step, followed by reorthogonalization using working-precision QR factorization and one-sided Jacobi sweeps \cite{GaoMS25}. Furthermore, in the solution of Sylvester matrix equations \cite{DmyFHetal25} and the computation of the matrix square root \cite{GaoKS26}, Schur-based initial approximations obtained from the low-precision QR-related kernels can be refined to working-precision accuracy. Finally, the QR-based CP-ALS algorithm may benefit from a stable QR factorization in half precision, given the satisfactory performance observed when switching the QR computation from double to single precision \cite[Fig.~5]{ZhaVV26}.

In this work, we propose a new mixed-precision variant of block Householder QR factorization, mpWYQR, to maximize low-precision computation without compromising overall factorization quality. Identifying the components of the algorithm most critical to accuracy, we examine which parts of the classical block approach can be computed at lower precision---such as inner products---and use higher precision to preserve numerical stability.  

The rest of the paper is organized as follows. We first provide background on floating-point arithmetic, rounding error analysis, and the block Householder QR decomposition with the WY representation. In Section~\ref{sec:variants}, we introduce our reduced- and mixed-precision block Householder QR variants, and in Section~\ref{sec:err}, we present their error analysis. We then present the numerical experiments in Section~\ref{sec:exp} and conclude the paper in Section~\ref{sec:conclusion}.
\section{Preliminaries}\label{sec:background}
This section introduces the concepts and methods necessary for the analysis and algorithms presented in the rest of the paper. We first review the basic concepts of floating-point arithmetic and finite-precision error analysis, which are then generalized to facilitate our mixed-precision analysis. Next, we describe the Householder QR factorization and discuss how blocking and the WY representation have been employed to improve its performance in scientific computing. Finally, we review the mixed-precision block Householder QR algorithm mpBQR3 proposed in \cite{YanFS21} and discuss the limitations that motivate the present work.

\subsection{Floating-point arithmetic and error analysis}\label{sec:fp-error-analysis}
Floating-point operations satisfy the standard model~\cite[sect.~2.2]{Hig02}
\[
\fl(a \circ b) = (a \circ b)(1+\delta), \qquad |\delta| \le u,
\]
with $u$ denoting the working precision of the operation; in this paper $u = \uhigh$ or $\ulow$ and $\fl_h$ or $\fl_l$, depending on the precision used, and we particularly focus on the precision combination $(\uhigh,\ulow) =(\mbox{fp32}, \mbox{fp16})$. 

The standard product lemma~\cite[Lem.~3.1]{Hig02} expresses the accumulation of rounding errors in uniform precision as
\[
 \prod_{i=1}^{k}(1+\delta_i)^{\rho_i}=1+\theta_k,
 \qquad |\theta_k|\le\gamma_k:=\frac{ku}{1-ku},
\]
where $|\delta_i|\le u$, $\rho_i\in\{-1,1\}$, and $ku<1$; see also~\cite[Chap.~1]{Wilk23}.
To distinguish the contributions of the two precisions and analyze rounding errors in mixed precision for mpWYQR later in Section~\ref{sec:err}, we extend this notation by defining, for nonnegative integers $p,q$,
\[
1+\theta_{p,q}
   := \prod_{i=1}^{p}(1+\delta_i^{h})^{\rho_i}
      \prod_{j=1}^{q}(1+\delta_j^{\ell})^{\tau_j},
\quad
|\delta_i^{h}|\le \uhigh,\quad
|\delta_j^{\ell}|\le u_\ell,\quad
\rho_i,\tau_j\in\{-1,1\}.
\]
The notation satisfies
\[
(1+\theta_{p,q})(1+\theta_{r,s}) = 1+\theta_{p+r,q+s},
\quad
(1+\theta_{p,q})^{-1} = 1-\frac{\theta_{p,q}}{1+\theta_{p,q}}
= 1+\theta'_{p,q},
\]
where $\theta'_{p,q}$ denotes a possibly different quantity of the same form as $\theta_{p,q}$. 
We omit the prime in subsequent calculations, in accordance with the convention that each occurrence may have a different value.
With $a=(1-\uhigh)^p(1-\ulow)^q\in(0,1]$, it follows that
$a\le1+\theta_{p,q}\le a^{-1}$. Since $1-a\le a^{-1}-1$, we obtain
\begin{equation}\label{eq:eta-all-orders}
 \eta_{p,q}:=(1-\uhigh)^{-p}(1-\ulow)^{-q}-1,
 \qquad
 |\theta_{p,q}|\le \eta_{p,q},
\end{equation}
Thus $\eta_{p,q}$ plays the same role for $\theta_{p,q}$ as $\gamma_k$ does for $\theta_k$.
The bound~\eqref{eq:eta-all-orders} retains all orders of unit roundoff and holds under the assumption $\uhigh,\ulow<1$.

Define 
\[
 \mu_{p,q}:=p\uhigh+q\ulow.
\]
If $\mu_{p,q}<1$, Bernoulli's inequality gives $(1-\uhigh)^p(1-\ulow)^q \ge1-\mu_{p,q}>0$. Taking reciprocals therefore yields
\[
 |\theta_{p,q}|\le\eta_{p,q}
 \le\frac{\mu_{p,q}}{1-\mu_{p,q}},
 \qquad \mu_{p,q}<1,
\]
where rightmost upper bound reduces to $\gamma_{p+q}$ when $\uhigh=\ulow=u$.
Furthermore, $(1-u)^{-1}\ge1+u$ implies $\eta_{p,q}\ge\mu_{p,q}$, so the preceding bound gives
\[
 0\le\eta_{p,q}-\mu_{p,q}
 \le\frac{\mu_{p,q}^{2}}{1-\mu_{p,q}},
 \qquad \mu_{p,q}<1.
\]
Consequently,
\[
 \eta_{p,q}=\mu_{p,q}+O(\mu_{p,q}^{2})
 \quad\text{as}\quad \mu_{p,q}\to0.
\]
The first-order estimate is useful only when $\mu_{p,q}\ll1$, just as the first-order approximation $\gamma_k \approx ku$ requires $ku\ll1$ in uniform precision.
We will also use
\[
 \widetilde\gamma_k^{(s)}
 :=\frac{c k u_s}{1-c k u_s},
 \qquad s\in\{h,\ell\},
\]
where $c$ is the modest positive constant and $c k u_s<1$.  

\subsection{Block Householder QR with WY representation}
The classical Householder QR algorithm factorizes an $m \times n$ matrix $A$ into an orthogonal matrix $Q$ and an upper triangular matrix $R$ by successively applying Householder reflections. At each step $k$, a Householder reflector $H_k = I - \tau_k v_k v_k^T$ zeros out the subdiagonal entries of the $k$-th column of $A$, where $v_k$ is the Householder vector and $\tau_k$ is the corresponding scalar. The reflector is then applied to the remaining columns of the matrix. After $n$ steps, the product $Q = H_1 H_2 \cdots H_n$ yields the orthogonal factor. While this approach is unconditionally backward stable, it relies heavily on BLAS2-level operations, which limits performance on modern architectures because of the large number of synchronizations and the inability to exploit cache efficiently.

The WY representation, introduced in \cite{BisV87}, reformulates the product of $k$ Householder reflectors as $I + WY^T$, where $W$ and $Y$ are $m \times k$ matrices. This representation allows the accumulated reflectors to be applied to a matrix block via (triangular) matrix-matrix products, namely TRMM and GEMM, rather than a sequence of BLAS2 operations. Since it involves more BLAS3 operations than the classical Householder QR approach, the WY representation is more advantageous for HPC applications.

To further improve performance, the algorithm is organized into blocks of size $b$. At a given block step, the next $b$ columns form a panel. The panel factorization constructs the corresponding $b$ Householder reflectors and accumulates them into a WY representation. These accumulated reflectors are then applied simultaneously to the remaining columns, called the trailing matrix. Compared with applying a single reflector at a time, the block formulation improves cache utilization by leveraging GEMM operations, providing better performance on large-scale problems. Algorithm \ref{alg:householder} presents the pseudocode for the block Householder QR approach with a WY representation. For more detailed algorithm, see \cite[Sec. 5]{GolV13}.

\begin{algorithm2e}
\caption{Block Householder QR with WY representation}
\label{alg:householder}
\KwIn{$A\in\mathbb{R}^{m\times n}$, $m\ge n$,
       and block size $b$.}
\KwOut{$R\in\mathbb{R}^{n\times n}$ and
        panel WY factors $\{(W_j,Y_j)\}_{j=1}^{s}$,
        where $s=\lceil n/b\rceil$.}

$j\gets 0$\;
\For{$k=1:b:n$}{
    $j\gets j+1$\;
    $r\gets\min(b,n-k+1)$\;
    $\ell\gets k+r-1$\;

    Factor the panel $A[k:m,k:\ell]$ by Householder QR, storing the 
    Householder vectors in $V$ and the corresponding scalars in $\tau$\; \label{algline:panelQR}
    Form the WY factors $W_j$ and $Y_j$ from $(V,\tau)$ so that the product of the panel reflectors is $I+W_jY_j^T$\;
    
    \If{$\ell<n$}{
        $C\gets A[k:m,\ell+1:n]$\;
        $A[k:m,\ell+1:n]
           \gets C+Y_j(W_j^T C)$\;
    }
}
$R\gets\operatorname{triu}(A[1:n,1:n])$\;
\end{algorithm2e}

\subsubsection{mpBQR3}
The mixed-precision blocked QR algorithm \texttt{mpBQR3} introduced in \cite[Alg.~7]{YanFS21} is based on a blocked WY approach in which each panel is first cast into $\uhigh$ and then factorized by Householder QR to build the WY representation in precision $\uhigh$. The resulting panel matrices, including the Householder basis and the WY factor, are subsequently cast back to $\ulow$ so that the trailing matrix update can be performed by low-precision matrix--matrix operations with $\uhigh$ accumulation, as in Tensor Core-type block fused multiply--add operations \cite{BlaHLetal20}.
This design exploits the high throughput of low-precision BLAS3 kernels while retaining $\uhigh$ for the lower-order operations inside the panel factorization. However, repeated cast-up and cast-down operations introduce additional expensive conversion overhead and low-precision perturbations, which cause the attainable accuracy to remain governed by $\ulow$, despite extensive use of $\uhigh$ inside each panel.

\section{New Reduced/Mixed Precision Variants}\label{sec:variants}
Revising the strategy above, we first investigate whether Householder QR can run uniformly in $\ulow$ without overflow or underflow while improving performance. We then introduce a mixed precision approach, mpWYQR, achieving $\bigO{\ulow}$ accuracy in terms of residual and LOO, with substantially fewer $\uhigh$ inner products and less precision-conversion overhead.

Algorithm~\ref{alg:householder} is presented using the general WY representation introduced by Bischof and Van Loan~\cite{BisV87}. 
In practice, a more storage-efficient compact WY representation is commonly used~\cite{SchV89}; with the sign convention adopted in our implementation, $Y=V$ and $W=YT$, so that only $V$ and the small triangular factor $T$ need to be stored. 
We retain the more general $W,Y$ notation in Algorithm~\ref{alg:householder} to keep the presentation independent of this particular implementation.

\subsection{Reduced precision variant (rWYQR)}
A natural starting point for a low-precision block Householder QR algorithm is to adapt the existing LAPACK \cite{Lapack} QR routines, \texttt{geqr2}, \texttt{geqrf}, and \texttt{geqrt}(\texttt{2},\texttt{3}), to operate in $\ulow$. However, a difficulty arises in the norm computation required for Householder reflector construction in line \ref{algline:panelQR} of Algorithm \ref{alg:householder}. 
LAPACK evaluates the Euclidean norm using a carefully scaled \texttt{nrm2} implementation to reduce the risk of intermediate overflow and underflow. Nevertheless, when the entire computation is carried out in a low precision such as fp16, the limited dynamic range still poses a numerical bottleneck for the norm evaluation. 
This motivates a more careful treatment of norm computation in the $\ulow$ setting.

To compute $\norm{x}_2$, we first evaluate the squared norm of the tail, $s = x_{2:n}^T x_{2:n} = \norm{x_{2:n}}_2^2$, and then form $\norm{x}_2 = \sqrt{x_1^2 + s}$ \cite{GolV13}. Thus, we avoid unnecessarily computing the intermediate quantity $\norm{x_{2:n}}_2 = \sqrt{s}$, saving one square-root operation and reducing the overflow and underflow risk.
After implementing the revised approach, we first use $\ulow$ throughout Algorithm \ref{alg:householder}, call it the rWYQR approach, and compare it with the \texttt{geqrf} routine in $\ulow$ for fp16. Table \ref{tab:rand3040k} shows the effects of norm scaling on time, LOO, and residuals. The test matrices $A \in \mathbb{R}^{m \times n}$ are generated with uniformly randomly distributed elements from the range $[-1, 1]$. All matrices are scaled from the right by the largest absolute value of each column to prevent overflow and/or underflow. We use a fixed block size of $n_b = 64$ throughout all experiments. We compare the reduced-precision variant with the proposed modification on norm computation, rWYQR, against the LAPACK routine for block Householder QR factorization analogous to mpBQR3, \texttt{geqrf}, in terms of LOO, $\|Q^TQ - I\|_F$, residual, $\|A - QR\|_F$, and runtime.

We observe that the residual and LOO are still very large, indicating that uniform fp16 arithmetic is insufficient for reliable factorization at these dimensions. However, norm scaling still yields more than a $\times 13$ speedup for a $30000\times 128$ matrix. When we increase the large dimension to $40000\times 100$, we still observe order $10^0$ error even with norm scaling, while still achieving more than $\times 12$ speedup. This result shows that with the introduced modification, reduced precision accelerates the computation of the Householder reflectors. However, while norm scaling is sufficient to prevent overflow, it is not sufficient to maintain full numerical accuracy to order $\ulow$.

\begin{table}[h]
\centering
\footnotesize
\begin{adjustbox}{max width=\textwidth}
\begin{tabular}{c c l l r}
\toprule
$(m,n)$ & Method & $||A-QR||_F$ & $||Q^TQ-I||_F$  & Time (s) \\
\midrule
\multirow{2}{*}{$(30000, 128)$}  & geqrf &  6.48926E-01 & \textbf{2.60498E-01} & 0.78886  \\
                        & rWYQR & \textbf{6.48438E-01} & \textbf{2.60498E-01} & \textbf{0.05832} \\

\midrule
\multirow{2}{*}{$(40000, 100)$}  & geqrf &  \textbf{2.54883E+00} & \textbf{2.75781E+00} & 0.68615  \\
                        & rWYQR &  2.80078E+00 & 3.39062E+00 & \textbf{0.05654} \\
\bottomrule
\end{tabular}
\end{adjustbox}
\vspace{0.1cm}
\caption{Performance and accuracy results for $A\in \mathbb{R}^{(m,n)}$ for various $(m,n)$ settings.}
\label{tab:rand3040k}
\end{table}

Motivated by the numerical experiments, we introduce a mixed-precision variant of the Householder construction given in Algorithm \ref{alg:householder-ref}. Since the two norm evaluations are the critical operations in $\ulow$, we perform both in $\uhigh$ and use $\ulow$ for the remaining operations. We choose the sign of the leading entry of the transformed vector according to the standard convention~\cite[eq.~(19.1)]{Hig02}. To show how the precision used for norm evaluations affects the accuracy of the algorithm, we run numerical experiments in two configurations. The first configuration is \texttt{rWYQR}, as explained above, where $\ulow = \uhigh$. In the second configuration, denoted \texttt{mpWYQR}, we implement Algorithm \ref{alg:householder-ref} with $\uhigh < \ulow$, corresponding to the mixed-precision implementation presented in Section \ref{sec:mpwyqr}.

\begin{algorithm2e}[h]
\caption{Mixed-precision Householder reflector construction}
\label{alg:householder-ref}
\KwIn{$x \in \mathbb{R}^m$}
\KwOut{$v\in \mathbb{R}^m$, $\beta\in \mathbb{R}$}
    $\alpha = -\sign(x_1) \cdot \norm{x}_2$\SComment{norm in $\uhigh$, result converted to $\ulow$}
    $v = (x-\alpha e_1) / \norm{x-\alpha e_1}_2$\SComment{norm in $\uhigh$, subtraction and division in $\ulow$}
    $\beta = 2 / (v^Tv)$\SComment{$\ulow$} \label{algline:beta}
\end{algorithm2e}

To investigate how the accuracy of the Householder reflector construction affects the overall QR factorization, we run experiments on square, dense random matrices of various sizes. Table \ref{tab:rand_sq_v32b32} shows relative errors in $v$ and $\beta$, computed using fp16 for the low-precision operations, against $v_{fp32}$ and $\beta_{fp32}$ computed in fp32. 
Although mpBQR3 gives more accurate $v$ and $\beta$ in most cases, the overall factorization errors remain similar. For instance, while the reported relative errors in $v$ for mpBQR3 and mpWYQR differ by $0.1135$ for $(m,n)=(4000,4000)$, their LOO and QR errors differ by at most $2\times10^{-4}$, which is smaller than one fp16 unit roundoff, $4.88\times10^{-4}$.
We anticipate that the reason for this behavior is that, during the calculation of $\beta$ in line \ref{algline:beta} of Algorithm \ref{alg:householder-ref}, the inaccuracy from the computation of $v_k$ from the previous line can be subject to cancellation due to the inner product computation $v_k^Tv_k$. This shows that using high precision for every inner product, as in mpBQR3, is unnecessary; instead, one can use $\uhigh$ to compute the norm, which can significantly improve performance. The importance of using $\uhigh$ for norm computation is also evident in the difference between the results of rWYQR and mpWYQR. We see that rWYQR cannot recover inaccuracies in computing $v$ and $\beta$, leading to highly inaccurate QR factors. The dashed entries in rWYQR show that the norm evaluations in $\ulow$ can easily overflow for large problems.

\begin{table}[h]
\centering \footnotesize
\begin{adjustbox}{max width=\textwidth}
\begin{tabular}{c c l l l l r}
\toprule
$(m,n)$ & Method & $\|A-QR\|_F$ & $\|Q^TQ-I\|_F$ & $\dfrac{\|v_{fp32}-v\|}{\|v_{fp32}\|}$ & $\dfrac{\|\beta_{fp32}-\beta\|}{\|\beta_{fp32}\|}$ & Time (s) \\
\midrule
\multirow{3}{*}{$(4000, 4000)$}   & rWYQR  & 8.55713E-02 & 9.95483E-02 & 0.67257 & 0.03590 & 0.33285 \\
                                  & mpWYQR & 5.99670E-03 & 7.92694E-03 & 0.63347 & 0.03217 & \textbf{0.31996} \\
                                  & mpBQR3 & \textbf{5.90515E-03} & \textbf{7.70950E-03} & \textbf{0.51997} & \textbf{0.02832} & 0.56634 \\
\midrule
\multirow{3}{*}{$(12000, 12000)$} & rWYQR  & 5.42969E-01 & 6.17676E-01 & 0.81345 & 0.02097 & 1.60879 \\
                                  & mpWYQR & 5.65338E-03 & 7.54547E-03 & \textbf{0.67193} & \textbf{0.01851} & \textbf{1.50179} \\
                                  & mpBQR3 & \textbf{5.54276E-03} & \textbf{7.40814E-03} & 0.68853 & 0.01968 & 5.09183 \\
\midrule
\multirow{3}{*}{$(36000, 36000)$} & rWYQR  & -- & -- & -- & -- & -- \\
                                  & mpWYQR & 3.51143E-03 & 5.52368E-03 & \textbf{0.72925} & 0.01332 & \textbf{30.14354} \\
                                  & mpBQR3 & \textbf{3.34358E-03} & \textbf{5.51224E-03} & 0.74434 & \textbf{0.01301} & 58.15787 \\
\midrule
\multirow{3}{*}{$(42000, 42000)$} & rWYQR  & -- & -- & -- & -- & -- \\
                                  & mpWYQR & \textbf{3.00980E-03} & 5.22614E-03 & \textbf{0.73835} & \textbf{0.01162} & \textbf{57.78656} \\
                                  & mpBQR3 & 3.15475E-03 & \textbf{5.15747E-03} & 0.74534 & 0.01252 & 79.60925 \\
\bottomrule
\end{tabular}
\end{adjustbox}
\caption{Performance and accuracy results for $A\in\mathbb{R}^{m\times n}$, where $m=n$.}
\label{tab:rand_sq_v32b32}
\end{table}

The numerical experiments show that the uniform $\ulow$ implementation achieves lower accuracy than $\bigO{u_{\text{fp16}}}$, indicating that norm evaluations in $\ulow$ introduce rounding errors that degrade the quality of the Householder reflectors. 

Error analysis in Section \ref{sec:err} shows that computing the norm in $\uhigh$ rather than $\ulow$ only decreases the constant in the error bound, while the structure of the heuristic error bound remains the same. This suggests the importance of using $\uhigh$ in the norm computation, especially in practice, as it stabilizes the computation without affecting the asymptotic accuracy of the algorithm.

\subsection{Mixed Precision variant (mpWYQR)}\label{sec:mpwyqr}
Building on the reduced-precision algorithm introduced in the previous subsection, we now present the mixed-precision variant, which addresses the numerical instabilities observed when uniform low precision is used throughout the algorithm.

As discussed above, the two 2-norm evaluations in the Householder reflector construction are particularly sensitive in low precision. We therefore use $\uhigh$ for both norm evaluations. 
The remaining parts of the algorithm are carried out in $\ulow$, and the resulting LOO remains of order $\bigO{\ulow}$.

Numerical experiments show that the Householder reflector $I - \beta vv^T$ determines accuracy.
In particular, if any precision $u_p$ is used to compute the Householder reflector, the accuracy of the resulting mixed-precision Householder QR with WY representation will be limited by the order $\bigO{u_p}$. This motivates a careful precision assignment: the norm computation in the reflector construction uses a $\uhigh$ accumulator, the resulting scalar is converted to $\ulow$, and the remaining reflector quantities are formed and stored in $\ulow$. The more computationally intensive operations, such as the matrix--matrix multiplications in the WY representation update, are then carried out in $\ulow$.

The overhead of mpWYQR over a uniform low-precision WY Householder QR is limited to the two norm evaluations used to form each Householder reflector. For the $j$th reflector, with active vector length $\ell_j=m-j+1$, each $\ell_2$-norm requires $2\ell_j-1$ high-precision additions and multiplications and one square root, followed by one scalar conversion to $\ulow$. Since the conversions occur in registers, they introduce no additional memory traffic or memory. Summed over all $n$ reflectors, the high-precision work is $2\sum_{j=1}^n(2\ell_j-1)=4mn-2n^2$ flops plus $2n$ square roots and $2n$ scalar conversions. This is lower order than the leading factorization cost $F(m,n)=2mn^2-\frac{2}{3}n^3$. In the tall-skinny regime ($m\gg n$), the relative high-precision work remains $\mathcal{O}(1/n)$.

We now compare the computational and communication costs of mpBQR3 and mpWYQR. Let $b=n_b$ be the block size, assume for simplicity that $n=Nb$, and let $m_k=m-(k-1)b$ be the active panel height at block step $k$. A key structural difference is that mpWYQR uses a \texttt{geqrt}-type blocked panel kernel in $\ulow$, whereas mpBQR3 uses a \texttt{geqrf}/HQR-type panel factorization and forms the WY representation in $\uhigh$ before casting it to $\ulow$. Both algorithms apply the block reflectors through WY updates; the main cost difference is therefore the precision used to generate the reflectors and WY factors.

In terms of computational cost, the panel QR and WY construction in mpBQR3 require $\mathcal{O}(m_kb^2)$ high-precision flops at step $k$, for a total of $\mathcal{O}\bigl(b^2\sum_{k=1}^N m_k\bigr)=\mathcal{O}(mnb)$ 
high-precision flops overall. Its trailing updates contain $F(m,n)+\mathcal{O}(mnb)$ low-input block-FMA (bFMA) flops, whose accumulation is performed in $\uhigh$. In contrast, mpWYQR restricts $\uhigh$ to norm computations, while the remaining, more computationally dominant operations are carried out in $\ulow$. Furthermore, the block structure of mpWYQR enables greater parallelization via batch computation. As a result, both algorithms have the same leading total flop count, but mpWYQR reduces the high-precision work from $\mathcal{O}(mnb)$ to $\mathcal{O}(mn)$.

mpWYQR also improves precision-conversion communication. Here we count only data movement caused by changing precision, since the ordinary blocked-QR traffic is common to both algorithms. In mpBQR3, casting each active panel up and the panel and WY factors back down requires a constant number of passes over $m_kb$ entries. Since $b\sum_{k=1}^N m_k=mn-\frac{1}{2}n(n-b)$, the additional conversion traffic is $\mathcal{O}(mn)$ array entries; the constant depends on storage reuse. In mpWYQR, the high-precision norm accumulators and the two converted scalars per reflector remain in registers, so no high-precision panel is materialized and the additional cost is only $\mathcal{O}(n)$ scalar register conversions, with no additional memory traffic. Overall, mpWYQR is cheaper in both computational and communication costs than mpBQR3; see Table \ref{tab:comp}.

\begin{table}[h]
    \centering
    \footnotesize
    \begin{adjustbox}{max width=\textwidth}
    \begin{tabular}{@{}l c c c@{}}
      \toprule
      Algorithm
      & $\uhigh$ cost
      & $\ulow$/bFMA cost
      & Conversion traffic\\
      \midrule
      mpBQR3
      & $\mathcal{O}(mnb)$
      & $F(m,n)+\mathcal{O}(mnb)$ bFMA
      & $\mathcal{O}(mn)$ array entries\\
      mpWYQR
      & $4mn-2n^2$ 
      & $F(m,n)+\mathcal{O}(mnb)$
      & $0$ array; $\mathcal{O}(n)$ register\\
      \bottomrule
    \end{tabular}
    \end{adjustbox}
    \vspace{0.1cm}
    \caption{Cost comparison of mpBQR3 and mpWYQR, where $F(m,n)=2mn^2-\frac{2}{3}n^3$.}
    \label{tab:comp}
\end{table}

\section{Error Analysis}\label{sec:err}
Let $A \in \mathbb{R}^{m \times n}$ with $m \ge n$. The Householder QR factorization triangularizes $A$ by applying a sequence of Householder matrices.
Given a nonzero vector $v \in \mathbb{R}^m$, the Householder matrix (or Householder reflector) is of the form
\begin{equation}\label{eq:householder-mat-def}
    H = I - \beta vv^{T}, \qquad 
    \beta = 2/ v^{T}v,
\end{equation}
which zeros out the elements of $A$ by columns. The Householder matrices are never formed explicitly in practice, because storage and computations are feasible by solely using the Householder vector $v$.

The following lemma bounds the forward errors in the
Householder vector and scaling factor computed by
Algorithm~\ref{alg:householder-ref}. Here and henceforth, computed quantities are denoted with a hat notation.

\begin{lemma}
\label{lem:householder-scaler-vec}
Let $x \in \mathbb{R}^m$ be nonzero, with $\sign(0)=1$, and define
\[
\alpha=-\sign(x_1)\norm{x}_2,\qquad
z=x-\alpha e_1,\qquad
v=\frac{z}{\norm{z}_2},\qquad
\beta=\frac{2}{v^Tv}=2.
\]
In floating-point arithmetic with $\uhigh$ and $\ulow$, the computed Householder vector $\wh{v}$ and scaling factor $\wh{\beta}$ produced by Algorithm~\ref{alg:householder-ref} satisfy
\begin{equation}\label{eq:householder-scaler-vec}
    |\wh v-v|\le \eta_{2m+2,4}|v|,
    \qquad
    |\wh\beta-\beta| \le \eta_{2m+2,m+5}|\beta|.
\end{equation}
In particular,
\[
\norm{\wh v-v}_2\le\eta_{2m+2,4}\norm{v}_2.
\]
Consequently, when the corresponding $\mu_{p,q}$ are small, the first-order estimates are
\[
\norm{\wh v-v}_2
\le \bigl((2m+2)\uhigh+4\ulow
      +O(\mu_{2m+2,4}^2)\bigr)\norm{v}_2,
\]
\[
|\wh\beta-\beta|
\le \bigl((2m+2)\uhigh+(m+5)\ulow
      +O(\mu_{2m+2,m+5}^2)\bigr)|\beta|.
\]

\begin{proof}
For the input vector $x$, the standard inner-product error expansion~\cite[sect.~3.1, eq.~(3.2)]{Hig02} and the nonnegativity of the summands give
\begin{equation}\label{eq:householder-high-dot}
 \fl_h(x^Tx)=\xi_h\norm{x}_2^2,
 \qquad (1-\uhigh)^m\le\xi_h\le(1+\uhigh)^m\le(1-\uhigh)^{-m}.
\end{equation}
Let $\wh r(x)$ denote the norm of $x$ evaluated by a high-precision dot product and square root, followed by conversion to low precision.
Taking the square root in~\eqref{eq:householder-high-dot} and including its rounding and the conversion error gives
\begin{equation}\label{eq:householder-norm-bound}
 a\norm{x}_2\le\wh r(x)\le a^{-1}\norm{x}_2,
 \qquad a:=(1-\uhigh)^{m+1}(1-\ulow).
\end{equation}
Thus $\wh\alpha=-\sign(x_1)\wh r(x)$ satisfies $a\le\wh\alpha/\alpha\le a^{-1}$.
Since the sign choice gives $|z_1|=|x_1|+|\alpha|>0$ and prevents cancellation, we also have $a\le(x_1-\wh\alpha)/z_1\le a^{-1}$.
For the subsequent subtraction, we have
\[
 \wh z_1=(x_1-\wh\alpha)(1+\delta_z^\ell),
 \qquad |\delta_z^\ell|\le\ulow.
\]
It follows that
\begin{equation}\label{eq:householder-z-update}
 \wh z_1=gz_1,\qquad \wh z_i=z_i\quad(i>1), \qquad
 g=(1+\delta_z^\ell)\frac{x_1-\wh\alpha}{z_1}.
\end{equation}
Combining the preceding bounds yields
\begin{equation}\label{eq:householder-z-factor}
 b:=a(1-\ulow)\le g\le a^{-1}(1-\ulow)^{-1}=b^{-1},
 \qquad b\in(0,1].
\end{equation}
In particular, $\wh z\ne0$.

To analyze the normalization step, define $q=\widehat z /\norm{\widehat z}_2$, the exact normalization of the computed vector, and $t= \norm{\widehat z}_2/ \norm{z}_2$.
By~\eqref{eq:householder-z-update}, we have $t^2=1+(g^2-1)v_1^2$.
Since $\norm{v}_2=1$, we have $0\le v_1^2\le1$, and hence
\[
 \min\{1,g^2\}\le t^2\le\max\{1,g^2\}.
\]
Taking positive square roots shows that $t$ lies between $g$ and $1$.
Hence $g/t$ lies between $g$ and $1$, while $1/t$ lies between $g^{-1}$ and $1$.
For $v_i\ne0$, combining this with~\eqref{eq:householder-z-factor} gives
\begin{equation}\label{eq:householder-q-ratio}
 b\le\frac{q_i}{v_i} 
 =\frac{\widehat z_i/\|\widehat z\|_2}{z_i/\|z\|_2} 
 =\begin{cases}g/t,&i=1,\\1/t,&i>1,\end{cases}
 \le b^{-1}.
\end{equation}
Zero components of $v$ remain zero in $q$.
Repeating the norm analysis in~\eqref{eq:householder-norm-bound} for $\wh z\ne0$ gives
\begin{equation}\label{eq:householder-norm-factor}
 c:=\wh r(\wh z) / \norm{\wh z}_2, \qquad a\le c\le a^{-1}.
\end{equation}
Then the low-precision divisions satisfy
\begin{equation}\label{eq:householder-vector-division}
 \wh v_i=\frac{\wh z_i}{\wh r(\wh z)}(1+\delta_i^\ell)
 =c^{-1}q_i(1+\delta_i^\ell),
 \qquad |\delta_i^\ell|\le\ulow.
\end{equation}
Collecting these errors in $D=\diag(1+\delta_1^\ell,\ldots,1+\delta_m^\ell)$ yields
\[
 \wh v=c^{-1}Dq,\qquad |D-I|\le\ulow I.
\]
For $v_i\ne0$, combining~\eqref{eq:householder-q-ratio}--\eqref{eq:householder-vector-division} and then using~\eqref{eq:householder-z-factor} gives
\[
 b^2 = ab(1-\ulow)\le\frac{\wh v_i}{v_i}
 =\frac{1+\delta_i^\ell}{c}\frac{q_i}{v_i}
 \le\bigl[ab(1-\ulow)\bigr]^{-1} = b^{-2}.
\]
Since $1-b^2\le b^{-2}-1$ and zero components are preserved,
\[
 |\wh v-v|\le(b^{-2}-1)|v|=\eta_{2m+2,4}|v|,
\]
where $b^2=(1-\uhigh)^{2m+2}(1-\ulow)^4$.

For the final dot product, the low-precision analogue of~\eqref{eq:householder-high-dot} gives
\begin{equation}\label{eq:householder-beta-dot}
 \fl_\ell(\wh v^T\wh v)=\xi_\ell\norm{\wh v}_2^2,
 \qquad (1-\ulow)^m\le\xi_\ell\le(1+\ulow)^m\le(1-\ulow)^{-m}.
\end{equation}
Set $s=\norm{Dq}_2$.
Since $D$ is diagonal, $|D-I|\le\ulow I$ implies $\norm{D-I}_2\le\ulow$.
Together with $\norm{q}_2=1$, the reverse triangle inequality therefore gives
\[
 |s-1|\le\norm{(D-I)q}_2\le\ulow.
\]
Thus, using $\wh v=c^{-1}Dq$, we have
\begin{equation}\label{eq:householder-beta-norm}
 1-\ulow\le s\le1+\ulow\le(1-\ulow)^{-1},
 \qquad \norm{\wh v}_2^2=c^{-2}s^2.
\end{equation}
The final low-precision division and $\beta=2$ then give
\begin{equation}\label{eq:householder-beta-division}
 \wh\beta
 =\frac{2(1+\delta_\beta^\ell)}{\fl_\ell(\wh v^T\wh v)}\\
 =\beta\frac{(1+\delta_\beta^\ell)c^2}{\xi_\ell s^2},
 \qquad |\delta_\beta^\ell|\le\ulow.
\end{equation}
Combining~\eqref{eq:householder-norm-factor} and~\eqref{eq:householder-beta-dot}--\eqref{eq:householder-beta-division} gives
\[
 L_\beta\le\frac{\wh\beta}{\beta}\le L_\beta^{-1},
 \qquad L_\beta:=a^2(1-\ulow)^{m+3}
 =(1-\uhigh)^{2m+2}(1-\ulow)^{m+5}.
\]
Since $0<L_\beta\le1$ and $1-L_\beta\le L_\beta^{-1}-1$, it follows that
\[
 |\wh\beta-\beta|\le(L_\beta^{-1}-1)|\beta|
 =\eta_{2m+2,m+5}|\beta|.
\]
Finally, taking Euclidean norms in the componentwise bound for $\wh v$ and expanding $\eta_{p,q}$ as in Section~\ref{sec:fp-error-analysis} gives the remaining statements.
\end{proof}
\end{lemma}

Lemma~\ref{lem:householder-scaler-vec} moves the dimension-dependent errors in the two norm evaluations to $\uhigh$, while the low-precision term proportional to $m\ulow$ remains in the bound for $\wh\beta$, since the inner product $\wh v^T\wh v$ is still evaluated in $\ulow$.
Next, we provide rounding error bounds for the application of a Householder matrix to a vector.

\begin{lemma}
\label{thm:householder-application-mp}
Let $v,\beta$ define the exact Householder matrix $P = I - \beta v v^T \in \mathbb{R}^{m\times m}$ associated with a nonzero $x$ as in Lemma~\ref{lem:householder-scaler-vec}, and let $\wh v,\wh\beta$ be the computed quantities produced by Algorithm~\ref{alg:householder-ref}.
Let $b\in\mathbb{R}^m$, and consider the computed application
\[
\wh y
= \fl_\ell \bigl(b - \wh{\beta}\,\wh{v}(\wh{v}^Tb)\bigr),
\]
where all arithmetic in the application is performed in $\ulow$. Then
\begin{equation}\label{eq:householder-application-mp}
\wh y=(P+\Delta P)b,
\qquad
\norm{\Delta P}_F\le\varepsilon_m,
\end{equation}
where the all-orders bound is
\begin{align}
\varepsilon_m := {}&
2\eta_{0,m+1}+4\eta_{m+1,4}+\eta_{0,1}
+2(1+\eta_{0,m+1})\eta_{0,m+3}.
\label{eq:epsilon-m}
\end{align}
When $\mu_{4m+4,4m+25}$ is small, this gives the first-order estimate
\begin{equation}\label{eq:epsilon-m-first-order}
\varepsilon_m
=(4m+4)\uhigh+(4m+25)\ulow
+O(\mu_{4m+4,4m+25}^2).
\end{equation}
\end{lemma}

\begin{proof}
Set $\widetilde P=I-\wh\beta\wh v\wh v^T$ and separate the error in the stored reflector from the rounding error in its application as
\begin{equation}\label{eq:householder-application-split}
 \wh y-Pb=(\widetilde P-P)b+(\wh y-\widetilde Pb).
\end{equation}
We first bound $\widetilde P-P$ using the quantities $q,c,D,s$ from the proof of Lemma~\ref{lem:householder-scaler-vec}.
Define
\[
 p=\frac{Dq}{s},\qquad
 \lambda=\frac{1+\delta_\beta^\ell}{\xi_\ell}.
\]
Since $s=\norm{Dq}_2>0$, we have $\norm{p}_2=1$.
By~\eqref{eq:householder-beta-dot} and~\eqref{eq:householder-beta-division},
\[
 (1-\ulow)^{m+1}\le\lambda\le(1-\ulow)^{-m-1},
 \qquad |\lambda-1|\le\eta_{0,m+1}.
\]
Moreover, $\wh v=(s/c)p$ and $\wh\beta=2\lambda c^2/s^2$, so
\begin{equation}\label{eq:householder-stored-reflector}
 \wh\beta\wh v\wh v^T=2\lambda pp^T,
 \qquad |\wh\beta|\norm{\wh v}_2^2=2|\lambda|.
\end{equation}

To compare $p$ with $v$, first note the division-error bounds in~\eqref{eq:householder-vector-division} and the bounds on $s$ in~\eqref{eq:householder-beta-norm} give
\[
 (1-\ulow)^2\le\frac{d_{ii}}{s}\le(1-\ulow)^{-2},
\]
where $d_{ii}$ denote the diagonal entries of $D$.
For $v_i\ne0$, we have $p_i/v_i=(d_{ii}/s)(q_i/v_i)$.
Combining this with~\eqref{eq:householder-z-factor} and~\eqref{eq:householder-q-ratio} yields
\[
 (1-\uhigh)^{m+1}(1-\ulow)^4
 \le\frac{p_i}{v_i}
 \le(1-\uhigh)^{-m-1}(1-\ulow)^{-4}.
\]
Zero components of $v$ remain zero in $p$, and hence
\[
 |p-v|\le\eta_{m+1,4}|v|,
 \qquad \norm{p-v}_2\le\eta_{m+1,4}.
\]
Using $\norm{p}_2=\norm{v}_2=1$ and $pp^T-vv^T=p(p-v)^T+(p-v)v^T$,
we obtain from~\eqref{eq:householder-stored-reflector}
\begin{align}
 \norm{\widetilde P-P}_F
 &=2\norm{\lambda pp^T-vv^T}_F
 \le2|\lambda-1|+4\norm{p-v}_2\notag\\
 &\le2\eta_{0,m+1}+4\eta_{m+1,4}.
 \label{eq:householder-construction-error}
\end{align}

We now hold $\wh v$ and $\wh\beta$ fixed and analyze their low-precision application to $b$.
For the dot product, we have
\begin{equation}\label{eq:householder-application-dot}
 d=\wh v^Tb,\qquad \wh d=\fl_\ell(\wh v^Tb),\qquad
 |\wh d-d|\le\eta_{0,m}|\wh v|^T|b|.
\end{equation}
In particular, $|\wh d|\le(1+\eta_{0,m})|\wh v|^T|b|$.
The standard componentwise error bounds for the two multiplications and the final subtraction~\cite[sect.~2.2]{Hig02} give
\begin{equation}\label{eq:householder-application-update}
 \bigl|\wh y-(b-\wh\beta\wh v\wh d)\bigr|
 \le\ulow|b|+|\wh\beta|\eta_{0,3}|\wh v|\,|\wh d|.
\end{equation}
Combining~\eqref{eq:householder-application-dot} and~\eqref{eq:householder-application-update}, and using
\[
 \eta_{0,m}+(1+\eta_{0,m})\eta_{0,3}=\eta_{0,m+3},
 \qquad \ulow\le\eta_{0,1},
\]
yields
\[
 |\wh y-\widetilde Pb|
 \le\eta_{0,1}|b|
 +|\wh\beta|\eta_{0,m+3}|\wh v|\bigl(|\wh v|^T|b|\bigr).
\]
Taking Euclidean norms and applying the Cauchy--Schwarz inequality gives
\[
 \norm{\wh y-\widetilde Pb}_2
 \le\left(\eta_{0,1}
 +|\wh\beta|\norm{\wh v}_2^2\eta_{0,m+3}\right)\norm{b}_2.
\]
By~\eqref{eq:householder-stored-reflector},
$|\wh\beta|\norm{\wh v}_2^2=2|\lambda|\le2(1+\eta_{0,m+1})$, and therefore
\begin{equation}\label{eq:householder-application-rounding}
 \norm{\wh y-\widetilde Pb}_2
 \le\left(\eta_{0,1}
 +2(1+\eta_{0,m+1})\eta_{0,m+3}\right)\norm{b}_2.
\end{equation}
Combining~\eqref{eq:householder-application-split}, \eqref{eq:householder-construction-error}, and~\eqref{eq:householder-application-rounding} gives $\norm{\wh y-Pb}_2\le\varepsilon_m\norm{b}_2$, with $\varepsilon_m$ defined in~\eqref{eq:epsilon-m}. For $b\ne0$, set
\[
 \Delta P=\frac{(\wh y-Pb)b^T}{\norm{b}_2^2}.
\]
Then $\Delta Pb=\wh y-Pb$. The identity $\norm{uv^T}_F=\norm{u}_2\norm{v}_2$, together with the preceding estimate, gives the required Frobenius-norm bound and proves~\eqref{eq:householder-application-mp}.
For $b=0$, the computed result is zero and we take $\Delta P=0$.
Finally, when $\mu_{4m+4,4m+25}\ll1$, expanding the $\eta_{p,q}$ terms in~\eqref{eq:epsilon-m}, as in Section~\ref{sec:fp-error-analysis}, gives~\eqref{eq:epsilon-m-first-order}.
\end{proof}

Next we derive columnwise error bounds for applying a sequence of Householder transformations to a matrix. These bounds describe the reflector-level sequential application; we discuss additional errors from forming and applying a WY representation below.

\begin{lemma}
\label{thm:mp-householder-columnwise}
Let $A\in\mathbb{R}^{m\times n}$ and consider $r$ Householder transformations
\[
P_k = I - \beta_k v_k v_k^T, \qquad k=1,\dots,r,
\]
whose computed reflector quantities are produced by Algorithm~\ref{alg:householder-ref}. Let $Q^T=P_rP_{r-1}\cdots P_1$. Then the computed result satisfies
\[
\wh A_{r+1}=Q^T(A+\Delta A),
\]
with the all-orders columnwise and normwise bounds
\begin{equation}\label{eq:columnwise-all-orders}
\norm{\Delta a_j}_2
\le\bigl((1+\varepsilon_m)^r-1\bigr)\norm{a_j}_2,
\qquad
\norm{\Delta A}_F
\le\bigl((1+\varepsilon_m)^r-1\bigr)\norm{A}_F.
\end{equation}
If $r\mu_{4m+4,4m+25}$ is small, the corresponding first-order estimate is
\[
\norm{\Delta a_j}_2
\le
\left(r\bigl((4m+4)\uhigh+(4m+25)\ulow\bigr)
+O(r^2\mu_{4m+4,4m+25}^2)\right)\norm{a_j}_2,
\]
and the same factor bounds $\norm{\Delta A}_F/\norm{A}_F$.
\end{lemma}

\begin{proof}
For each column $a_j$, Lemma~\ref{thm:householder-application-mp} gives
\[
\wh a_j^{(r+1)}
=(P_r+\Delta P_{r,j})\cdots(P_1+\Delta P_{1,j})a_j,
\qquad
\norm{\Delta P_{k,j}}_2\le\norm{\Delta P_{k,j}}_F\le\varepsilon_m.
\]
Since every $P_k$ is orthogonal,
\[
\norm{
(P_r+\Delta P_{r,j})\cdots(P_1+\Delta P_{1,j})
-P_r\cdots P_1
}_2
\le (1+\varepsilon_m)^r-1.
\]
Setting $\Delta a_j=Q\wh a_j^{(r+1)}-a_j$ proves the columnwise bound. Summing the squared columnwise bounds gives the Frobenius-norm bound. The first-order estimate follows from~\eqref{eq:epsilon-m-first-order} and the expansion of $(1+\varepsilon_m)^r$.
\end{proof}

For the active column $x$ used to construct a reflector, we have $Px=\alpha e_1$, and the stored diagonal entry $\wh\alpha$ satisfies $|\wh\alpha-\alpha|\le\eta_{m+1,1}\norm{x}_2\le\varepsilon_m\norm{x}_2$. Thus Lemma~\ref{thm:mp-householder-columnwise} yields the following QR backward-error result by the standard argument~\cite[Thm.~19.4]{Hig02}.

\begin{theorem}
\label{thm:mp-householder-stad}
Let $\wh R\in\R^{m\times n}$ be the computed upper trapezoidal QR factor of $A\in\R^{m\times n}$ obtained from the Householder reflector constructions and sequential applications analyzed above. Then there exists an orthogonal $Q\in\R^{m\times m}$ such that
\[
    A + \Delta A = Q \wh R,
\]
where
\begin{equation}\label{eq:qr-backward-all-orders}
\norm{\Delta a_j}_2
\le \bigl((1+\varepsilon_m)^n-1\bigr)\norm{a_j}_2,
\qquad j=1,\ldots,n.
\end{equation}
If $n\mu_{4m+4,4m+25}$ is small, the first-order estimate is
\[
\norm{\Delta a_j}_2
\le
\left(n\bigl((4m+4)\uhigh+(4m+25)\ulow\bigr)
+O(n^2\mu_{4m+4,4m+25}^2)\right)\norm{a_j}_2.
\]
The matrix $Q$ is given explicitly by $Q=(P_nP_{n-1}\cdots P_1)^T$, where $P_k$ is the exact Householder matrix associated with the $k$th active column.
\end{theorem}

Finally, we present results concerning the error in the computed $Q$, the factorization residual, and the loss of orthogonality.

\begin{theorem}
\label{thm:mp-householder-loo}
Under the assumptions of Theorem~\ref{thm:mp-householder-stad}, let $\wh Q^T$ be the result of applying the same computed reflector quantities to $I_m$ in the order used to form $Q^T$, and set $\wh Q=(\wh Q^T)^T$. Then
\begin{equation}\label{eq:q-error-all-orders}
\norm{\wh Q-Q}_F
\le \sqrt{m}\bigl((1+\varepsilon_m)^n-1\bigr).
\end{equation}
Moreover,
\begin{equation}\label{eq:residual-all-orders}
\norm{A-\wh Q\wh R}_F
\le
\bigl((1+\varepsilon_m)^n-1\bigr)
\bigl(1+\sqrt{m}(1+\varepsilon_m)^n\bigr)\norm{A}_F,
\end{equation}
and
\begin{equation}\label{eq:loo-all-orders}
\norm{\wh Q^T\wh Q-I}_F
\le
2\sqrt{m}\bigl((1+\varepsilon_m)^n-1\bigr)
+m\bigl((1+\varepsilon_m)^n-1\bigr)^2.
\end{equation}
If $n\mu_{4m+4,4m+25}$ is small, the respective first-order estimates are
\[
\norm{\wh Q-Q}_F
\le
n\sqrt{m}\bigl((4m+4)\uhigh+(4m+25)\ulow\bigr)
+O(n^2\sqrt{m}\,\mu_{4m+4,4m+25}^2),
\]
\[
\norm{A-\wh Q\wh R}_F
\le
n(1+\sqrt{m})\bigl((4m+4)\uhigh+(4m+25)\ulow\bigr)\norm{A}_F
+O\bigl(n^2(1+\sqrt{m})\mu_{4m+4,4m+25}^2\bigr)\norm{A}_F,
\]
and
\[
\norm{\wh Q^T\wh Q-I}_F
\le
2n\sqrt{m}\bigl((4m+4)\uhigh+(4m+25)\ulow\bigr)
+O(mn^2\mu_{4m+4,4m+25}^2).
\]
The orthogonal matrix $Q$ is the one in Theorem~\ref{thm:mp-householder-stad}.
\begin{proof}
Applying Lemma~\ref{thm:mp-householder-columnwise} to $I_m$ and transposing the result gives~\eqref{eq:q-error-all-orders}. Let
\[
\Delta A=Q\wh R-A,
\qquad
E=\wh Q-Q.
\]
By Theorem~\ref{thm:mp-householder-stad} and~\eqref{eq:q-error-all-orders},
\[
\norm{\Delta A}_F
\le\bigl((1+\varepsilon_m)^n-1\bigr)\norm{A}_F,
\]
\[
\norm{E}_F
\le\sqrt{m}\bigl((1+\varepsilon_m)^n-1\bigr),
\qquad
\norm{\wh R}_F
\le(1+\varepsilon_m)^n\norm{A}_F.
\]
Therefore,
\[
\norm{A-\wh Q\wh R}_F
\le\norm{\Delta A}_F+\norm{E}_F\norm{\wh R}_F,
\]
which gives~\eqref{eq:residual-all-orders}. Finally,
\[
\wh Q^T\wh Q-I=Q^TE+E^TQ+E^TE,
\]
so
\[
\norm{\wh Q^T\wh Q-I}_F
\le2\norm{E}_F+\norm{E}_F^2,
\]
which proves~\eqref{eq:loo-all-orders}. Expanding the all-orders bounds gives the stated first-order estimates.
\end{proof}
\end{theorem}


Note that the mixed-precision bounds in Theorems~\ref{thm:mp-householder-stad} and~\ref{thm:mp-householder-loo} have the same qualitative form as the standard uniform-precision error bounds~\cite[sect.~19.3]{Hig02}.
Our results essentially reduce these bounds by a constant factor via performing the norm evaluation in the Householder matrix construction at $\uhigh$.
Since our approach differs from the standard construction of the Householder vector and scaling factor only in the norm evaluation, the above conclusion also applies when the WY representation~\cite{BisV87}, or its more storage-efficient compact WY representation~\cite{SchV89}, is used in the assembly and application of the Householder matrices~\cite[sect.~19.5]{Hig02}.
Additional rounding errors incurred in forming and applying the WY representation must be accounted for by the corresponding matrix--matrix error analysis.

Connolly and Higham~\cite[Thm.~4.4 and eq.~(4.20)]{ConH23} show that the dimension dependence in the standard deterministic Householder QR bounds can be substantially reduced under a probabilistic model in which rounding errors are mean independent and have mean zero, together with suitable probabilistic bounds for applying a sequence of Householder matrices to a vector.
In uniform precision, the worst-case $mnu$ dependence is replaced by $\sqrt{mn}u$, up to constants and powers of the probability parameter $\lambda$~\cite[eq.~(4.17)]{ConH23}.
An analogous probabilistic loss-of-orthogonality result also holds~\cite[Thm.~4.5]{ConH23}.
Applying the same reasoning to the present mixed-precision analysis suggests that the deterministic first-order term
\[
n\bigl((4m+4)\uhigh+(4m+25)\ulow\bigr)
\]
may be replaced, schematically, by terms of order $\sqrt{mn}(\uhigh+\ulow)$. A rigorous probabilistic result for the present precision assignment would require a separate analysis, but this viewpoint helps explain why the deterministic bounds can be much more pessimistic than the observed residuals and loss of orthogonality.

In comparison, the mixed-precision analysis of Yang, Fox, and Sanders~\cite[Alg.~7 and eqs.~(4.26)--(4.27)]{YanFS21} studies a different precision assignment. In mpBQR3, each panel is cast to $\uhigh$, the panel QR factorization and WY representation are formed in $\uhigh$, and the WY factors are cast back to $\ulow$ so that the trailing updates can use low-precision matrix products with higher-precision accumulation. If $r$ is the block size and $N=n/r$ is the number of column blocks, their error bound for the computed thin orthogonal factor has the leading form
\[
    \norm{\wh Q_{\mathrm{mpBQR3}}-Q}_F
    \lesssim
    \sqrt{n}\bigl(\widetilde\gamma_N^{(\ell)}
    + n\widetilde\gamma_m^{(h)}\bigr).
\]
Since $\widetilde\gamma_N^{(\ell)}\approx CN\ulow$ and $\widetilde\gamma_m^{(h)}\approx Cm\uhigh$, the low-precision contribution comes mainly from block-level cast-down errors, while the panel factorization contributes through the high-precision term.

Our mpWYQR variant uses $\ulow$ more aggressively: only the two norm evaluations in the reflector construction are carried out in $\uhigh$, while the remaining reflector construction, reflector application, and WY-style updates are governed by $\ulow$. Consequently, the deterministic bounds in Theorems~\ref{thm:mp-householder-stad} and~\ref{thm:mp-householder-loo} do not predict high-precision accuracy; they predict $\ulow$-level accuracy with the norm computations protected from overflow, underflow, and large accumulation errors that can occur in a fully low-precision reflector construction. The comparison therefore highlights a trade-off. mpBQR3 obtains a sharper high-precision panel term by doing substantially more panel work and conversion in $\uhigh$, whereas mpWYQR keeps the high-precision work to the lower-order norm evaluations and aims for a cheaper factorization whose attainable accuracy is still expected to be of order $\ulow$.
\section{Numerical Experiments}\label{sec:exp}
The primary goals of our implementation are to leverage existing LAPACK and BLAS routines, minimize hardware-specific dependencies to enable a native CPU implementation, and achieve performance improvements in large-scale experiments. However, several challenges arise in practice. Standard LAPACK and BLAS libraries do not provide routines for lower-precision arithmetic; mixed-precision routines are largely limited to GPU libraries, and memory bandwidth can become a performance bottleneck.

To overcome these challenges, we build on the work of K{\"o}hler and Benner \cite{KoeBen26}, which provides implementations of fp16 BLAS3 operations as well as mixed-precision LAPACK routines supporting fp16-fp32 precision combination. These routines form the foundation for the mpWYQR implementation.

The implementation is carried out using Intel Math Kernel Library (MKL), which provides native support for fp16 \texttt{GEMM} operations, and is executed on an Intel Xeon Platinum Emerald Rapids Series CPU, which natively supports fp16 arithmetic. This setup lets us evaluate the proposed algorithm on a standard CPU architecture without relying on GPU-specific libraries, making the implementation more broadly accessible. The code for generating the results in this section is available at \cite{KoeOLetal26}.

\subsection{Random well-scaled matrices}
We first test mpWYQR on well-scaled random matrices of various sizes to examine accuracy and performance behavior. We generate the test matrices $A \in \mathbb{R}^{m \times n}$ with uniformly random elements in $[-1, 1]$. All matrices are scaled from the right to prevent overflow and/or underflow. A fixed block size of $n_b = 64$ is used throughout all experiments, and the precision pair $(\ulow, \uhigh) = (\text{fp16}, \text{fp32})$ is used. We ran each experiment 5 times and report the average results throughout this section. 

We compare the proposed algorithm against \texttt{sgeqrt}, which is the standard block Householder QR factorization in fp32, rWYQR, and mpBQR3 in terms of LOO, $\|Q^TQ - I\|_F$, residual, $\|A - QR\|_F$, and runtime.

Table \ref{tab:rand_sqr} shows experiments on square random matrices with sizes ranging from $m=n=1000$ to 45000. For small matrices ($m=n\leq 2000$), we observe that \texttt{sgeqrt} or mpBQR3 can outperform mpWYQR with higher accuracy. However, as the size increases, mpWYQR starts to outperform the other variants, achieving more than $2.4\times$ speedup over \texttt{sgeqrt} and more than $3\times$ speedup over mpBQR3. As a trade-off, mpWYQR provides less accurate results than \texttt{sgeqrt}, as expected. This error difference follows from the error analysis as well: mpWYQR guarantees an error up to $\bigO{\ulow}$, whereas \texttt{sgeqrt} error is bounded by $\bigO{\uhigh} = \bigO{u_{fp32}}\approx 10^{-7}$. Compared with mpBQR2, mpWYQR yields slightly less accurate results, again as expected. However this difference is between $3.36\times 10^{-4}$ and $2.29\times 10^{-5}$, which can be neglected considering that both algorithms are using fp16 as the working precision and unit roundoff of fp16 is $4.88\times 10^{-4}$. Comparing rWYQR with mpWYQR shows the importance of using a high-precision accumulator when computing the norm. Both errors are higher in rWYQR than in mpWYQR. Another important issue is the performance of rWYQR. Although rWYQR uses uniform fp16 and mpWYQR uses fp32 elements, mpWYQR still slightly outperforms rWYQR. This behavior is related to the compiler and hardware properties. Even though rWYQR uses fp16 BLAS routines, the compiler is still underdeveloped for native fp16 operations. The hardware (CPU in this case) supports only basic operations ($+,-,\times,\backslash$), and, unfortunately, performs advanced operations in fp32. These limitations show the importance of mixed precision and can also motivate hardware and compiler development for lower precisions. We expect performance to improve for rWYQR when GPUs with Tensor Cores are used. However, one should note that the error is still expected to be worse than mpWYQR. The dashed lines in rWYQR show algorithm failure due to fp16 overflow.

\begin{table}[h]
\centering 
\begin{adjustbox}{max width=\textwidth}
\begin{tabular}{c c r r r r r r}
\toprule
Method & & $(1000,1000)$ & $(2000,2000)$ & $(4000,4000)$ & $(10000,10000)$ & $(26000,26000)$ & $(45000,45000)$ \\
\midrule
\multirow{3}{*}{sgeqrt} & $\|A-QR\|_F$   & \textbf{7.92003E-07} & \textbf{9.82351E-07} & \textbf{1.27310E-06} & \textbf{1.86841E-06} & \textbf{3.42155E-06} & \textbf{5.54632E-06} \\
                        & $\|Q^TQ-I\|_F$ & \textbf{1.02469E-06} & \textbf{1.25968E-06} & \textbf{1.59754E-06} & \textbf{2.32236E-06} & \textbf{4.17879E-06} & \textbf{6.72147E-06} \\
                        & Time (s)        & 0.19498 & \textbf{0.16725} & 0.35354 & 2.34586 & 24.34429 & 134.03649 \\
\midrule
\multirow{3}{*}{rWYQR} & $\|A-QR\|_F$   & 1.70898E-02 & 3.21960E-02 & 8.55713E-02 & 3.93311E-01 & --- & --- \\
                        & $\|Q^TQ-I\|_F$ & 2.04315E-02 & 3.78113E-02 & 9.95483E-02 & 4.41895E-01 & --- & --- \\
                        & Time (s)        & 0.22967 & 0.22002 & 0.33285 & 1.24347 & --- & --- \\
\midrule
\multirow{3}{*}{mpWYQR} & $\|A-QR\|_F$   & 5.87082E-03 & 5.98907E-03 & 5.99670E-03 & 5.83267E-03 & 4.18854E-03 & 2.92206E-03 \\
                        & $\|Q^TQ-I\|_F$ & 7.68280E-03 & 7.94220E-03 & 7.92694E-03 & 7.79343E-03 & 6.08826E-03 & 5.12695E-03 \\
                        & Time (s)        & 0.11587 & 0.20817 & \textbf{0.31996} & \textbf{1.17185} & \textbf{9.98710} & \textbf{72.41314} \\
\midrule
\multirow{3}{*}{mpBQR3} & $\|A-QR\|_F$   & 5.53513E-03 & 5.84412E-03 & 5.90515E-03 & 5.70679E-03 & 4.31061E-03 & 2.96021E-03 \\
                        & $\|Q^TQ-I\|_F$ & 7.08008E-03 & 7.59888E-03 & 7.70950E-03 & 7.72095E-03 & 6.06537E-03 & 5.05066E-03 \\
                        & Time (s)        & \textbf{0.07670} & 0.21081 & 0.56634 & 3.59873 & 28.64521 & 92.30898 \\
\bottomrule
\end{tabular}
\end{adjustbox}
\caption{Performance and accuracy results for $A\in\mathbb{R}^{m\times n}$, where $m=n$.}
\label{tab:rand_sqr}
\end{table}

Now, we perform the same comparison using random tall-skinny matrices. Table \ref{tab:rand_ts} shows that the behavior of all routines using fp16 changes with tall-skinny matrices. Now, when $m\leq 10000$ or $m\geq 330000$, \texttt{sgeqrt} outperforms the other routines. This is because the MKL fp16 \texttt{GEMM} routine is optimized only for square matrices. Therefore, using fp16 does not provide the expected performance benefit. We observe that \texttt{sgeqrt} can be more than $1.5\times$ faster than mpWYQR, especially when $m$ is small. As $m$ increases, the speedup drops to around $1.06\times$. In contrast, when $10000<m<330000$, mpWYQR achieves more than $5.8\times$ speedup over \texttt{sgeqrt}. The comparison between mpBQR3 and mpWYQR, however, remains similar to the square matrices case. mpWYQR outperforms mpBQR3 for all experiments, achieving up to $9\times$ speedup with at  most $1.67 \times 10^{-5}$ difference in accuracy. Last but not least, we observe the effect of fp32 norm computation even for $m=40000$, when rWYQR fails due to overflow and mpWYQR continues.

\begin{table}[h]
\centering 
\begin{adjustbox}{max width=\textwidth}
\begin{tabular}{c c r r r r r r}
\toprule
Method & & $(10000,128)$ & $(20000,128)$ & $(40000,128)$ & $(100000,128)$ & $(320000,128)$ & $(330000,128)$ \\
\midrule
\multirow{3}{*}{sgeqrt} & $\|A-QR\|_F$   & \textbf{2.76481E-07} & \textbf{2.98540E-07} & \textbf{3.59624E-07} & \textbf{5.89988E-07} & \textbf{1.89635E-06} & \textbf{1.96205E-06} \\
                        & $\|Q^TQ-I\|_F$ & \textbf{1.76102E-07} & \textbf{2.01938E-07} & \textbf{2.43083E-07} & \textbf{4.36624E-07} & \textbf{1.35620E-06} & \textbf{1.39062E-06} \\
                        & Time (s)        & \textbf{0.04000} & 0.09250 & 0.08043 & 0.25604 & 0.19918 & \textbf{0.14789} \\
\midrule
\multirow{3}{*}{rWYQR} & $\|A-QR\|_F$   & 6.62231E-02 & 2.82715E-01 & -- & -- & -- & -- \\
                        & $\|Q^TQ-I\|_F$ & 4.08630E-02 & 1.14807E-01 & -- & -- & -- & -- \\
                        & Time (s)        & 0.06461 & 0.02027 & -- & -- & -- & -- \\
\midrule
\multirow{3}{*}{mpWYQR} & $\|A-QR\|_F$   & 1.10340E-03 & 8.78811E-04 & 7.93457E-04 & 7.59125E-04 & 7.36237E-04 & 6.69003E-04 \\
                        & $\|Q^TQ-I\|_F$ & 1.82986E-04 & 1.45435E-04 & 9.25660E-05 & 7.43270E-05 & 5.67436E-05 & 7.03335E-05 \\
                        & Time (s)        & 0.06049 & \textbf{0.01582} & \textbf{0.02771} & \textbf{0.05797} & \textbf{0.16621} & 0.15726 \\
\midrule
\multirow{3}{*}{mpBQR3} & $\|A-QR\|_F$   & 9.64642E-04 & 7.63416E-04 & 6.63757E-04 & 6.51836E-04 & 6.09398E-04 & 6.52313E-04 \\
                        & $\|Q^TQ-I\|_F$ & 1.90377E-04 & 1.46985E-04 & 8.95858E-05 & 7.54595E-05 & 5.94854E-05 & 7.33137E-05 \\
                        & Time (s)        & 0.06201 & 0.14174 & 0.17388 & 0.44792 & 0.82904 & 0.79343 \\
\bottomrule
\end{tabular}
\end{adjustbox}
\vspace{0.1cm}
\caption{Performance and accuracy results for $A\in\mathbb{R}^{m\times n}$ with $m>n=128$.}
\label{tab:rand_ts}
\end{table}

\subsection{Special block matrix classes}

We now run several numerical experiments to compare the performance and accuracy of the four routines using three matrix classes available in the \texttt{BlockStab} library \cite{LunOCetal24}: \texttt{glued}, \texttt{monomial}, and \texttt{piled}. Each matrix class is created using dimensional inputs $m,p,s$, where $m$ denotes the number of rows, $p$ denotes the number of block vectors, and $s$ denotes the number of columns in each block vector. 
We set $m = 100$, $p = 10$, and $s=10$ for \texttt{glued} matrices; $m = 1200$, $p = 120$, $s = 2:2:12$ for \texttt{monomial} matrices; and $m = 1000:1000:10000$, $p = 10$, $s = 10$ for \texttt{piled} matrices. We ran each experiment 5 times and present the average results in Figures \ref{fig:heatmaps_glued}--\ref{fig:heatmaps_piled}. 

We observe that mpWYQR and mpBQR3 have similar LOO, with mpWYQR yielding slightly worse LOO, especially when $m$ is small. However, this difference appears only in the fourth significant digit and is therefore not significant. The relative error behaves similarly, especially as the condition number increases, but as with the LOO, the difference is negligible. In contrast, mpWYQR is up to $18\times$ faster than mpBQR3 (Figure~\ref{fig:heatmaps_monomial}). Across all the figures, mpWYQR performs best when $m$ is moderately small (around 3000), regardless of the condition number, and relatively worse when $m$ is even smaller. 

As expected, sgeqrt yields the best LOO and relative error for the \texttt{glued} matrices in Figure~\ref{fig:heatmaps_glued}. Although we cannot explain the scattered high-error yellow boxes in these two maps, the absolute differences are less than $10^{-6}$ and can be ignored. We also point out the effect of fp16 on timing by comparing the maps of rWYQR and mpWYQR. This effect becomes more visible as the system grows moderately large. In particular, when $m \geq 8000$, comparing the map of mpBQR3 with those of rWYQR and mpWYQR shows the impact of using fp16 in inner products.

Figure~\ref{fig:heatmaps_monomial} most clearly shows why computing norms in fp32 matters. Once the system is sufficiently large, rWYQR fails to compute the QR decomposition due to overflow or underflow, which appears as white boxes in the map. In contrast, mpWYQR succeeds even when $m$ is very large and yields LOO and relative error similar to those of mpBQR3. This figure also shows the largest performance difference between mpBQR3 and mpWYQR, which occurs when the system is very large and $s$ is sufficiently small.

Figure~\ref{fig:heatmaps_piled} also shows a significant performance gain for mpWYQR over mpBQR3, with negligible changes in LOO and relative error. These results are consistent with our theory, which shows that the error bounds of mpBQR3 and mpWYQR differ only by a constant. However, since both algorithms heuristically have the same error bound and the differences in the results are negligible, we argue that our theory provides a worst-case bound, and that the constant difference does not appear in practice.

\begin{figure}[htbp]
    \centering
    \includegraphics[width=0.354\textwidth, trim=20 0 8 0, clip]{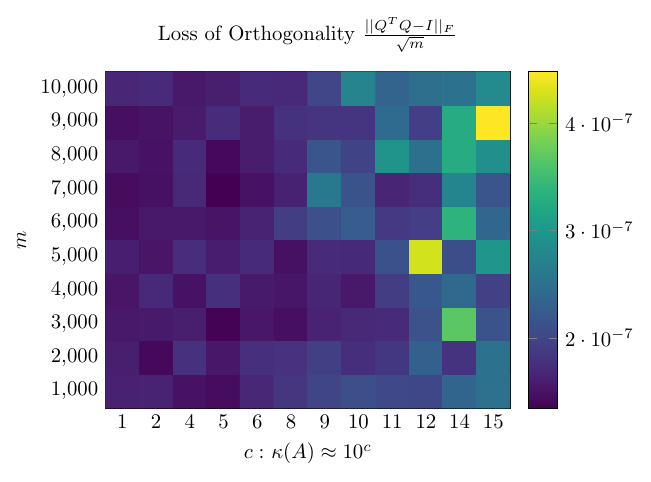}\hfill
    \includegraphics[width=0.317\textwidth, trim=50 0 8 0, clip]{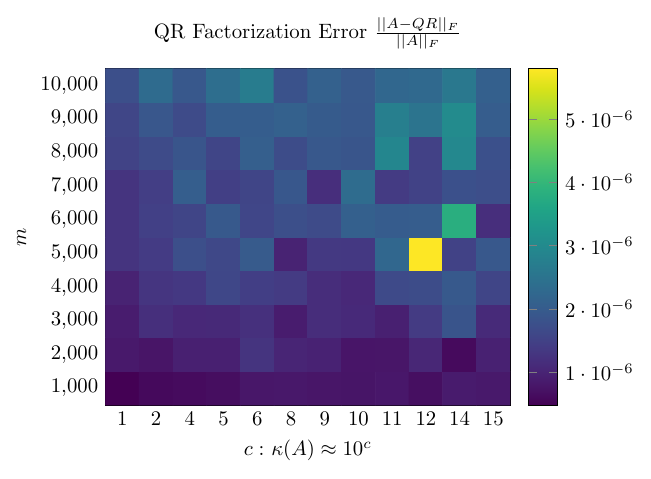}\hfill
    \includegraphics[width=0.32\textwidth, trim=50 0 8 0, clip]{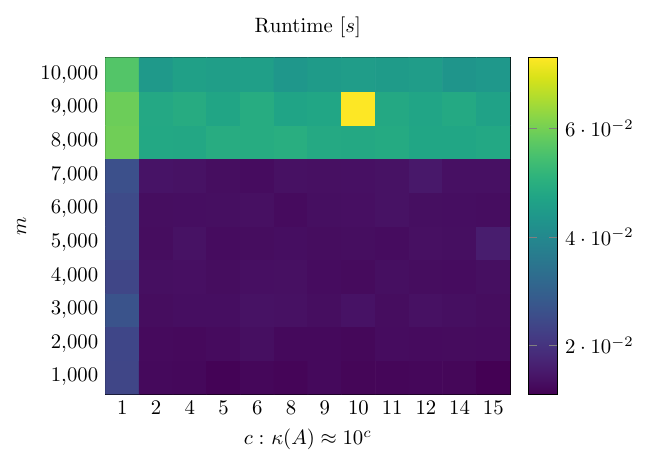}

    \vspace{0.5em}
    \includegraphics[width=0.354\textwidth, trim=20 0 32 0, clip]{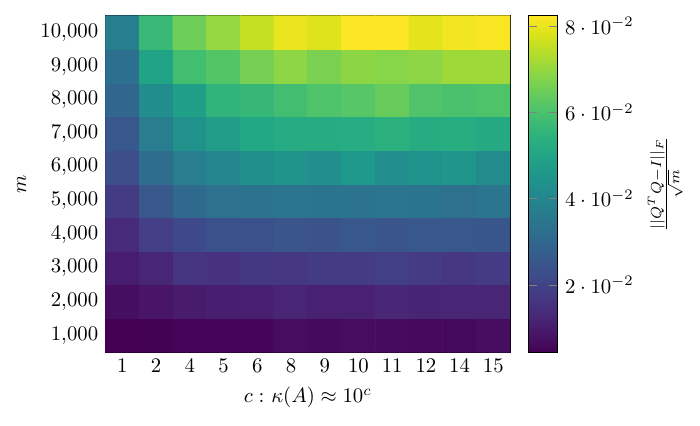}\hfill
    \includegraphics[width=0.317\textwidth, trim=50 0 31 0, clip]{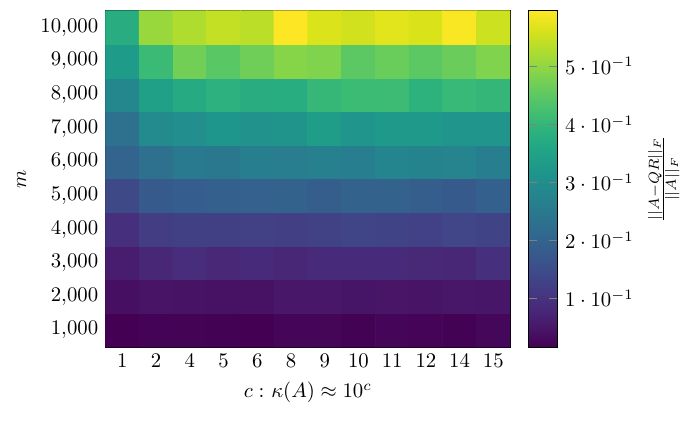}\hfill
    \includegraphics[width=0.32\textwidth, trim=50 0 25 0, clip]{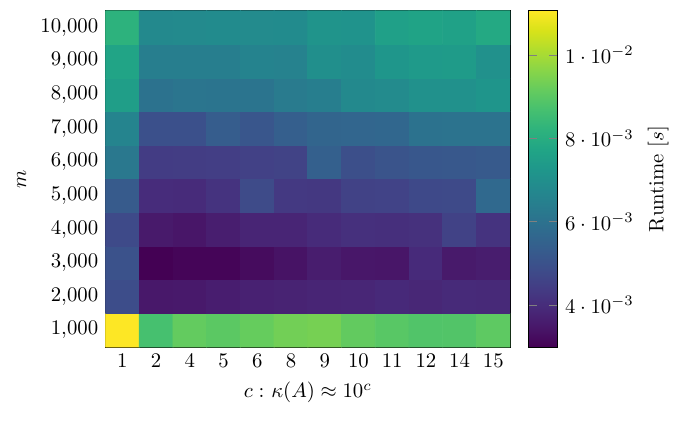}

    \vspace{0.5em}
    \includegraphics[width=0.354\textwidth, trim=20 0 32 0, clip]{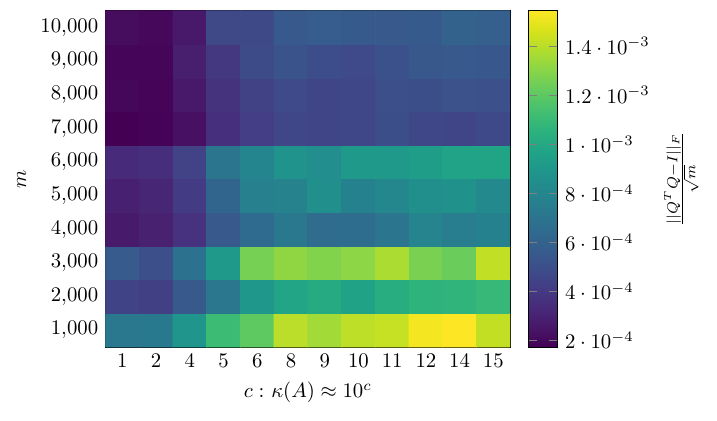}\hfill
    \includegraphics[width=0.317\textwidth, trim=50 0 31 0, clip]{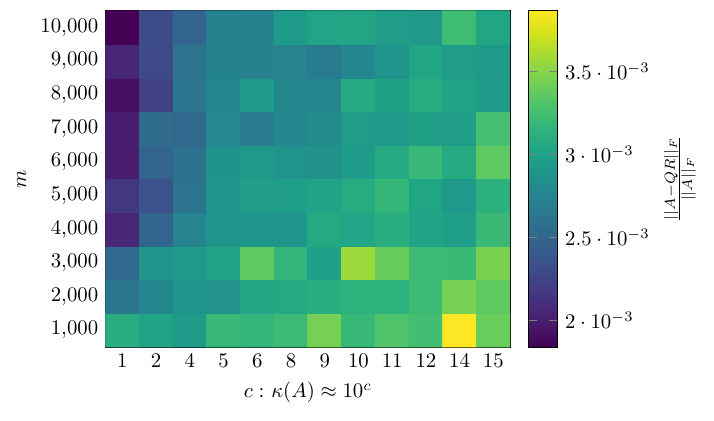}\hfill
    \includegraphics[width=0.32\textwidth, trim=50 0 25 0, clip]{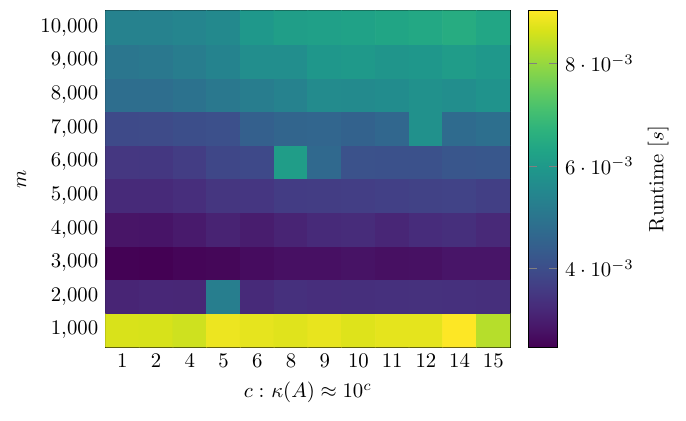}

    \vspace{0.5em}
    \includegraphics[width=0.354\textwidth, trim=20 0 32 0, clip]{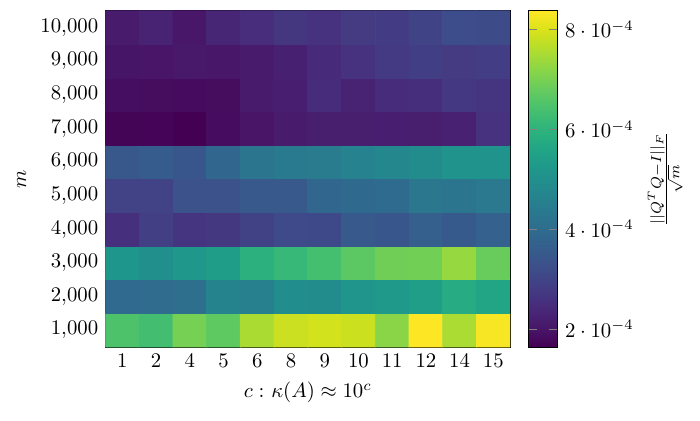}\hfill
    \includegraphics[width=0.317\textwidth, trim=50 0 31 0, clip]{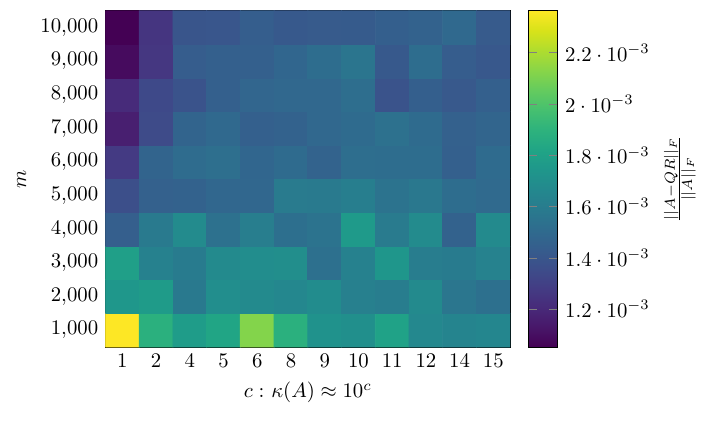}\hfill
    \includegraphics[width=0.32\textwidth, trim=50 0 25 0, clip]{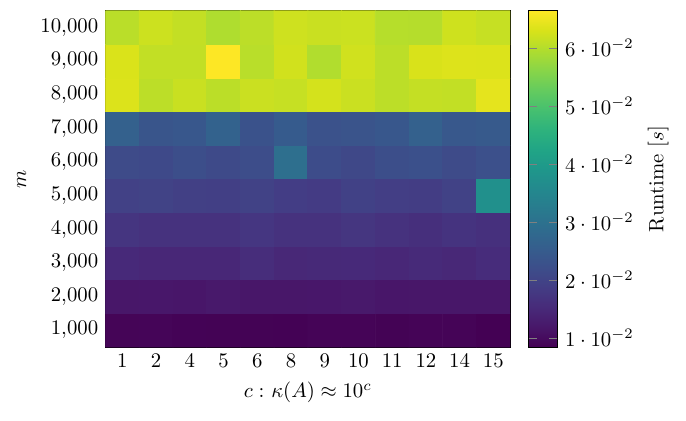}

    \caption{Loss of orthogonality (left), relative QR factorization error (middle), and runtime (right) for \texttt{glued} matrices with varying condition number and number of rows~$m$. From top to bottom, the rows correspond to sgeqrt, rWYQR, mpWYQR, and mpBQR3.}
    \label{fig:heatmaps_glued}
\end{figure}

\begin{figure}[htbp]
    \centering
    \includegraphics[width=0.354\textwidth, trim=20 0 32 0, clip]{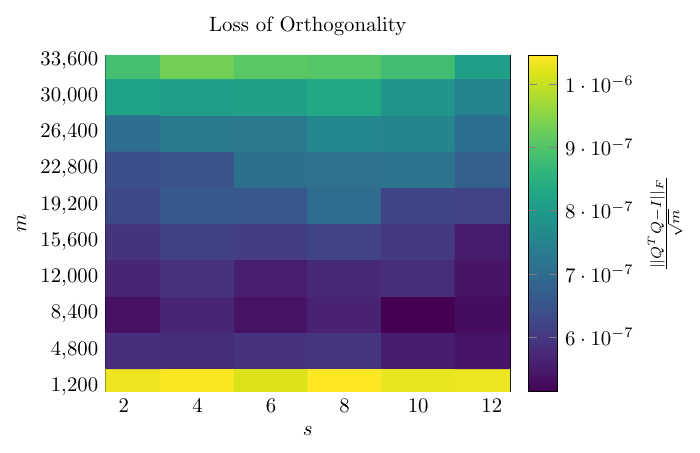}\hfill
    \includegraphics[width=0.317\textwidth, trim=50 0 8 0, clip]{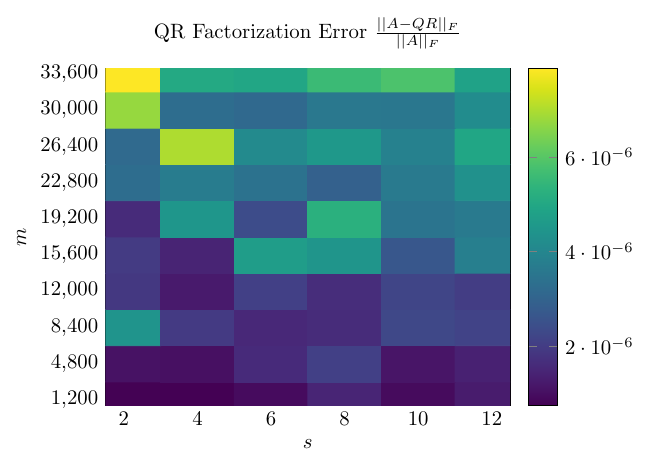}\hfill
    \includegraphics[width=0.32\textwidth, trim=50 0 8 0, clip]{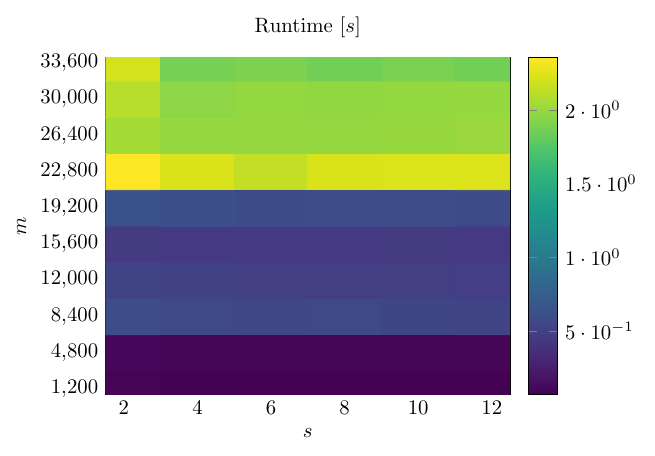}

    \vspace{0.5em}
    \includegraphics[width=0.354\textwidth, trim=20 0 32 0, clip]{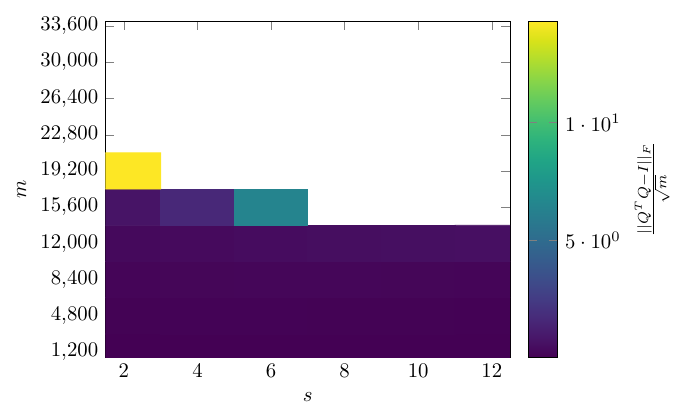}\hfill
    \includegraphics[width=0.317\textwidth, trim=50 0 31 0, clip]{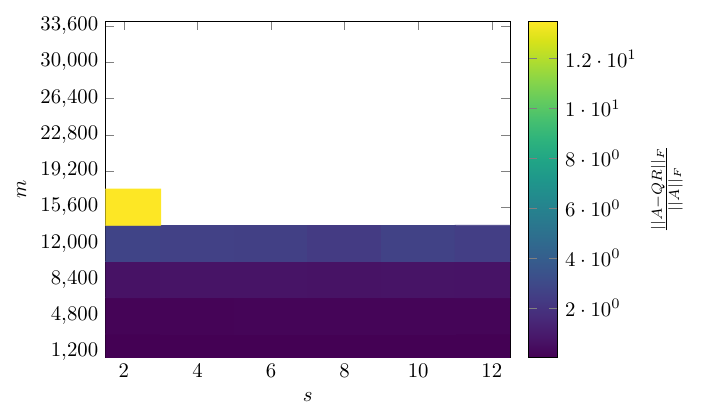}\hfill
    \includegraphics[width=0.32\textwidth, trim=50 0 25 0, clip]{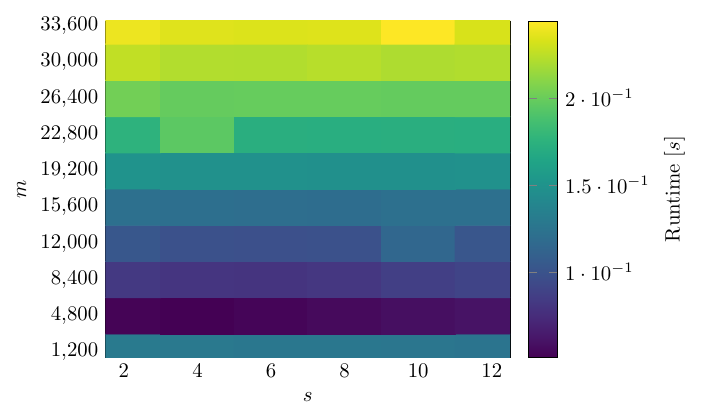}

    \vspace{0.5em}
    \includegraphics[width=0.354\textwidth, trim=20 0 32 0, clip]{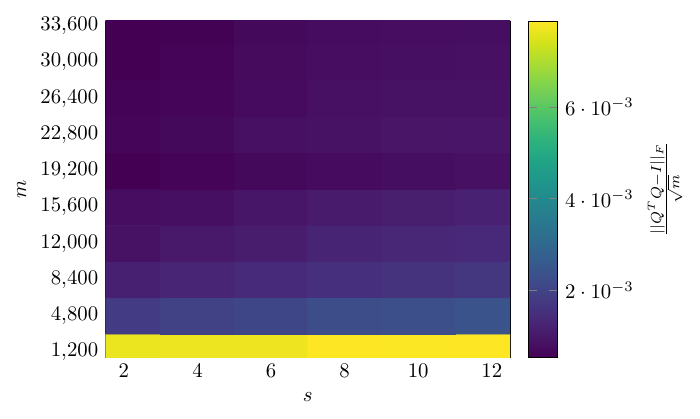}\hfill
    \includegraphics[width=0.317\textwidth, trim=50 0 31 0, clip]{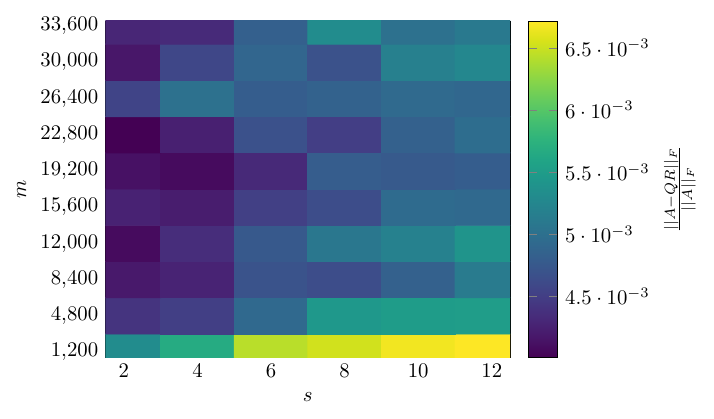}\hfill
    \includegraphics[width=0.32\textwidth, trim=50 0 25 0, clip]{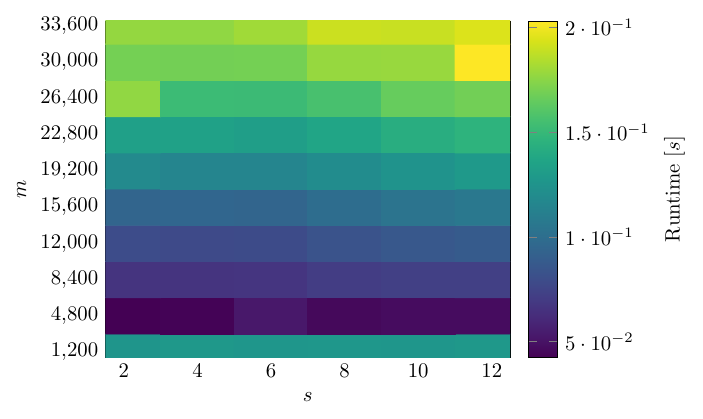}

    \vspace{0.5em}
    \includegraphics[width=0.354\textwidth, trim=20 0 32 0, clip]{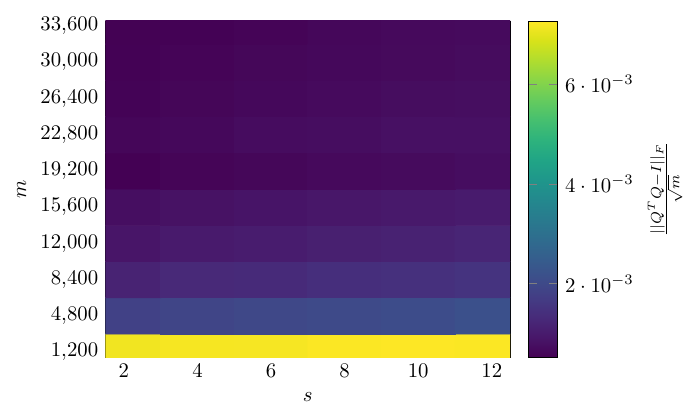}\hfill
    \includegraphics[width=0.317\textwidth, trim=50 0 31 0, clip]{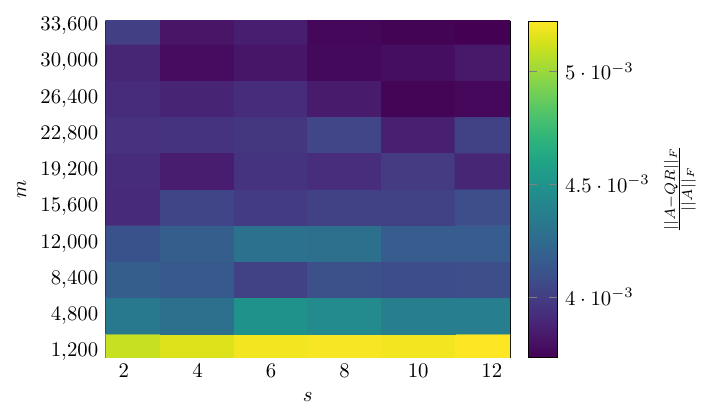}\hfill
    \includegraphics[width=0.32\textwidth, trim=50 0 25 0, clip]{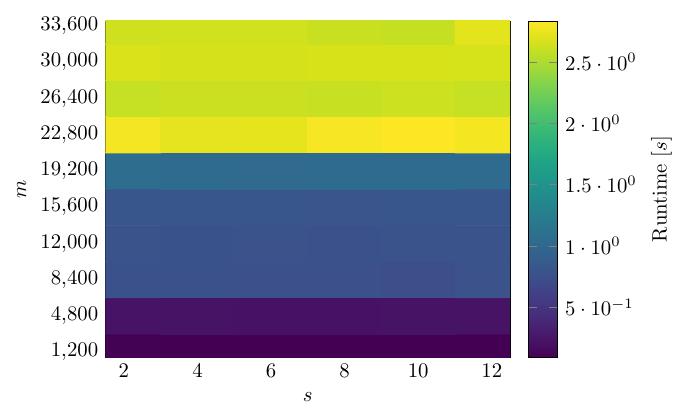}

    \caption{Loss of orthogonality (left), relative QR factorization error (middle), and runtime (right) for \texttt{monomial} matrices with varying block size $s$ and number of rows~$m$. From top to bottom, the rows correspond to sgeqrt, rWYQR, mpWYQR, and mpBQR3. The white areas appear when the result is NaN or Infinity.}
    \label{fig:heatmaps_monomial}
\end{figure}

\begin{figure}[htbp]
    \centering
    \includegraphics[width=0.35\textwidth, trim=21 0 30 0, clip]{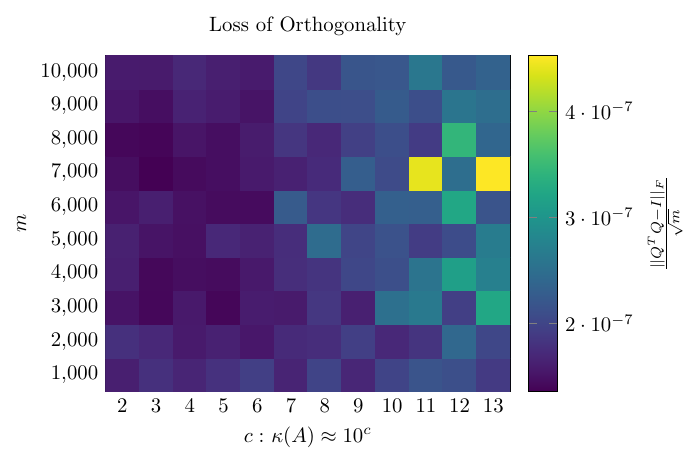}\hfill
    \includegraphics[width=0.325\textwidth, trim=50 0 8 0, clip]{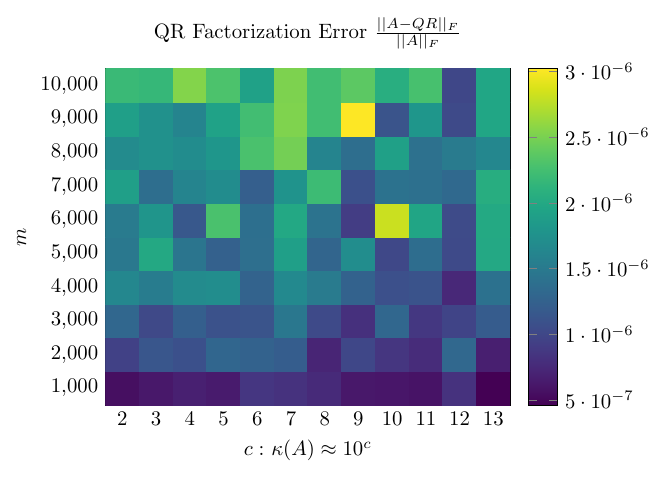}\hfill
    \includegraphics[width=0.315\textwidth, trim=50 0 8 0, clip]{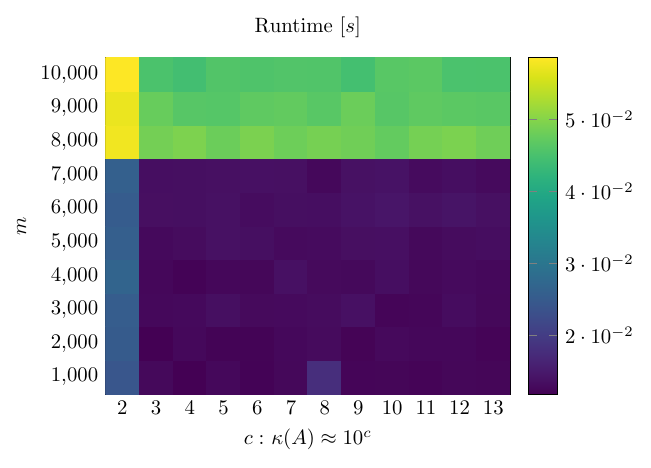}

    \vspace{0.5em}
    \includegraphics[width=0.35\textwidth, trim=21 0 33 0, clip]{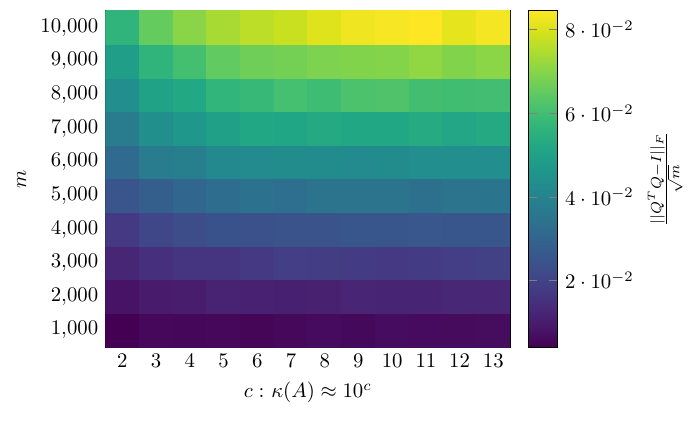}\hfill
    \includegraphics[width=0.325\textwidth, trim=50 0 31 0, clip]{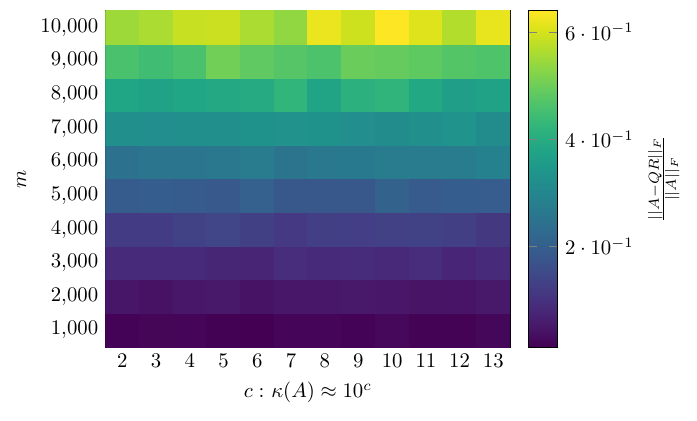}\hfill
    \includegraphics[width=0.315\textwidth, trim=50 0 25 0, clip]{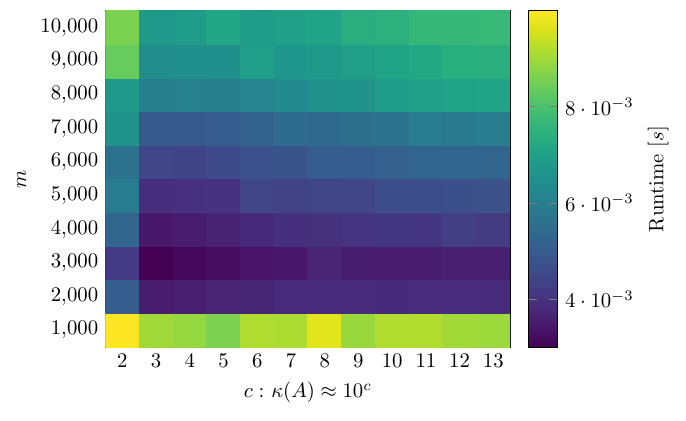}

    \vspace{0.5em}
    \includegraphics[width=0.35\textwidth, trim=21 0 33 0, clip]{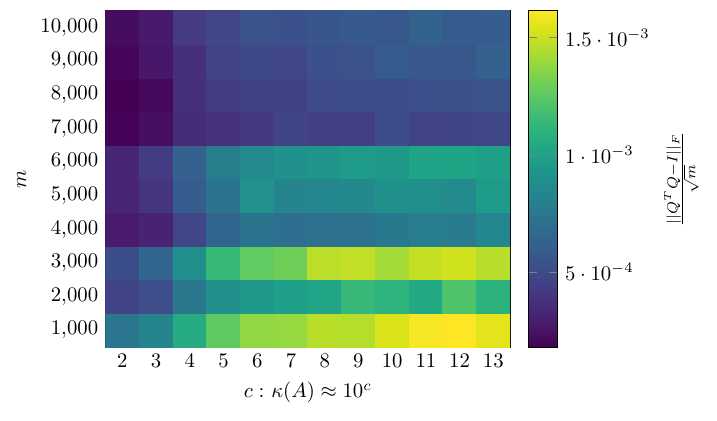}\hfill
    \includegraphics[width=0.325\textwidth, trim=50 0 31 0, clip]{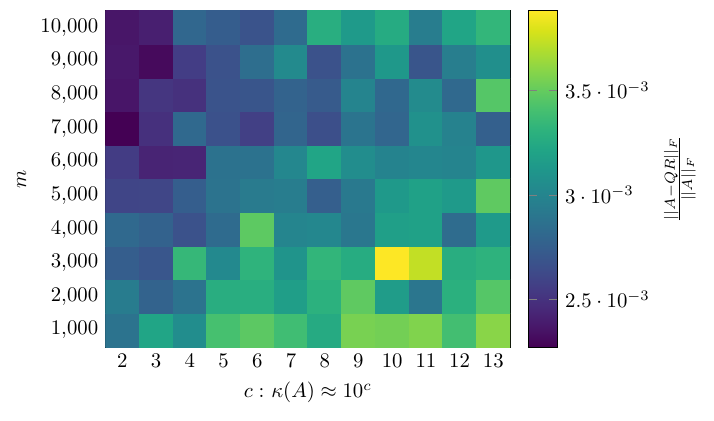}\hfill
    \includegraphics[width=0.315\textwidth, trim=50 0 25 0, clip]{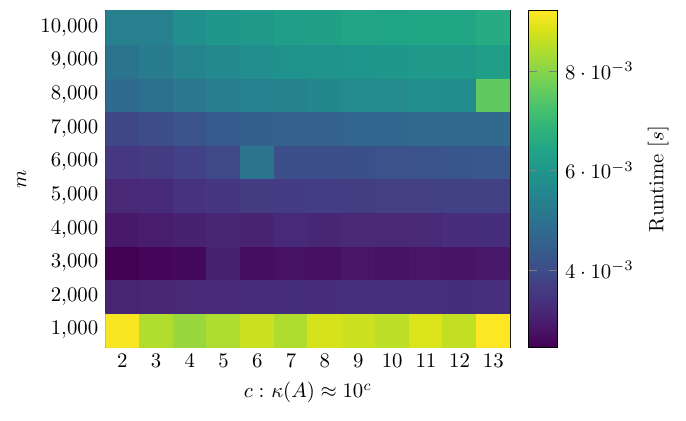}

    \vspace{0.5em}
    \includegraphics[width=0.35\textwidth, trim=21 0 33 0, clip]{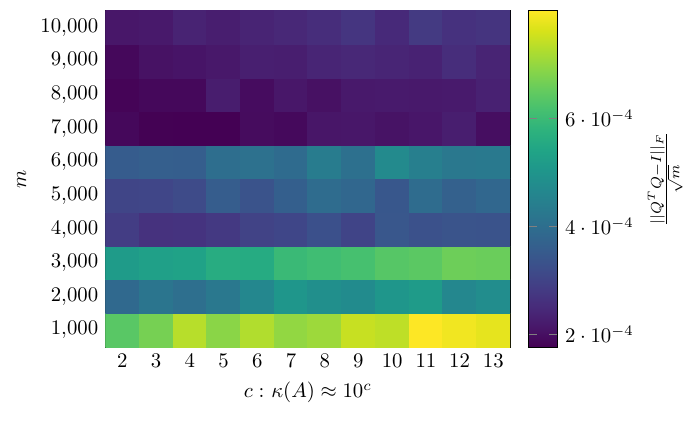}\hfill
    \includegraphics[width=0.325\textwidth, trim=50 0 31 0, clip]{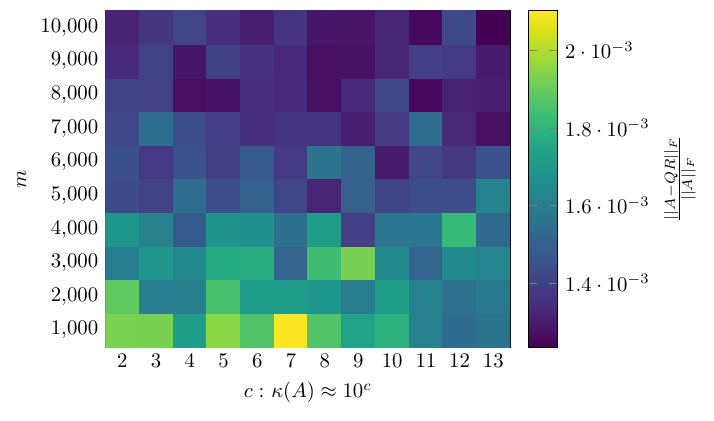}\hfill
    \includegraphics[width=0.315\textwidth, trim=50 0 25 0, clip]{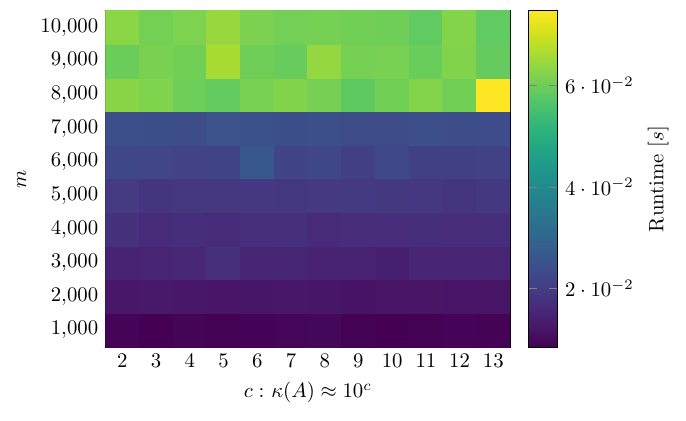}

    \caption{Loss of orthogonality (left), relative QR factorization error (middle), and runtime (right) for \texttt{piled} matrices with varying condition number and number of rows~$m$. From top to bottom, the rows correspond to sgeqrt, rWYQR, mpWYQR, and mpBQR3.}
    \label{fig:heatmaps_piled}
\end{figure}
\section{Conclusion} \label{sec:conclusion}

In this paper, we introduced a mixed-precision block Householder QR algorithm with WY representation, mpWYQR, and compared it with mpBQR3 introduced in \cite{YanFS21}. We also compared it with the uniform fp16 variant of mpWYQR, namely rWYQR, to show the importance of using fp32 accumulators in norm computations rather than using fp16 throughout the entire algorithm. In addition, we included sgeqrt as a control to show timing changes and compared it with the other algorithms. In all cases, mpBQR3 and mpWYQR behave similarly in terms of accuracy, while mpWYQR can provide significant speedup.

The accuracy of mpWYQR still depends on $u_l$, while the algorithm has fewer performance restrictions since it requires fewer precision conversions and performs more inner products in fp16, which increases the error bound by a constant. However, we showed that this constant does not significantly affect the relative error and loss of orthogonality compared with mpBQR3 in the numerical experiments, since both algorithms share the same heuristic bounds.


Our numerical experiments show that, on CPUs, the performance of our variants is currently limited by library support, since only GEMM is natively supported in fp16. With broader fp16 support in numerical libraries, our variants and fp16 algorithms in general can achieve better performance on CPUs. As future work, we plan to develop a GPU implementation of mpWYQR, which we anticipate will achieve comparable, if not better, performance while maintaining comparable accuracy.

\section*{Acknowledgments}%
\addcontentsline{toc}{section}{Acknowledgments}
Anthropic was used to structure the first draft of Sections 1-3. These sections were reviewed by all authors and substantially changed. No new ideas were contributed by Claude, and it was not used to generate or suggest a literature review or any references. All ideas in this paper belong to the authors, and we take full responsibility for the content.
E.O. is funded by the German Research Foundation (DFG) through the research training group 2297 “MathCoRe”, Magdeburg.

\addcontentsline{toc}{section}{References}

\bibliographystyle{abbrvurl}
\bibliography{references} 

\end{document}